\documentclass[11pt]{article}

\newcommand{\rats}{\mbox{$\mathbb Q$}}
\newcommand{\field}{\mbox{$\mathbb F$}}
\newcommand{\comps}{\mbox{$\mathbb C$}}
\newcommand{\nats}{\mbox{$\mathbb N$}}
\newcommand{\ints}{\mbox{$\mathbb Z$}}

\newcommand{\charac}{{\mathop{\mathrm{char}}\nolimits}}
\newcommand{\comment}[1]{}

\usepackage{amsmath, amsfonts, amssymb, latexsym, stmaryrd}
\usepackage{graphicx}                   

\def\squarebox#1{\hbox to #1{\hfill\vbox to #1{\vfill}}}
\def\qed{\hspace*{\fill}
        \vbox{\hrule\hbox{\vrule\squarebox{.667em}\vrule}\hrule}\smallskip}
\newenvironment{proof}{\begin{trivlist}
  \item[\hspace{\labelsep}{\em\noindent Proof.~}]
  }{\qed\end{trivlist}}

\newenvironment{psmallmatrix}
  {\left(\begin{smallmatrix}}
  {\end{smallmatrix}\right)}

\newtheorem{lemma}{Lemma}[section]
\newtheorem{theorem}[lemma]{Theorem}
\newtheorem{corollary}[lemma]{Corollary}
\newtheorem{proposition}[lemma]{Proposition}
\newtheorem{claim}[lemma]{Claim}
\newtheorem{observation}[lemma]{Observation}
\newtheorem{definition}[lemma]{Definition}

\newtheorem{question}[lemma]{Question}

\def\squareforqed{\hbox{\rlap{$\sqcap$}$\sqcup$}}
\def\qed{\ifmmode\squareforqed\else{\unskip\nobreak\hfil
\penalty50\hskip1em\null\nobreak\hfil\squareforqed
\parfillskip=0pt\finalhyphendemerits=0\endgraf}\fi}

\newlength{\tablength}
\newlength{\spacelength}
\newcommand{\tabstar}{\hspace*{\tablength}}
\newcommand{\spacestar}{\hspace*{\spacelength}}
\def\obeytabs{\catcode`\^^I=\active}
{\obeytabs\global\let^^I=\tabstar}
{\obeyspaces\global\let =\spacestar}
\newenvironment{display}{\begingroup\obeylines\obeyspaces\obeytabs}{\endgroup}
\newenvironment{prog}{\begin{display}\parskip0pt\sf}{\end{display}}

\title{On a special presentation of matrix algebras}

\author{
{\sl Geir Agnarsson}
\thanks{Department of Mathematical Sciences,
George Mason University,
MS 3F2,
4400 University Drive,
Fairfax, VA -- 22030, USA,
{\tt geir@math.gmu.edu}}
\and
{\sl Samuel S.~Mendelson}
\thanks{The Naval Surface Warfare Center,
Dahlgren Division,
6149 Suite 203, 
Welsh Rd, 
Dahlgren, VA -- 22448, USA,
{\tt samuel.mendelson@navy.mil}}
}

\date{}

\begin{document}

\maketitle

\begin{abstract}
Recognizing when a ring is a complete matrix ring is of significant 
importance in algebra.  It is well-known folklore that a ring $R$ is a 
complete $n\times n$ matrix ring, so $R\cong M_{n}(S)$ for some ring $S$, 
if and only if it contains a set of $n\times n$ matrix 
units $\{e_{ij}\}_{i,j=1}^n$. A more recent and less known result states 
that a ring $R$ is a complete $(m+n)\times(m+n)$ matrix ring
if and only if, $R$ contains three elements, $a$, $b$, 
and $f$, satisfying the two relations $af^m+f^nb=1$ and $f^{m+n}=0$. 
In many instances the two 
elements $a$ and $b$ can be replaced by appropriate powers $a^i$ and $a^j$
of a single element $a$ respectively. In general very little
is known about the structure of the ring $S$. In this article we study
in depth the case $m=n=1$ when $R\cong M_2(S)$. More specifically we
study the universal algebra over a commutative ring $A$ with elements
$x$ and $y$ that satisfy the relations $x^iy+yx^j=1$ and $y^2=0$.  
We describe completely the structure of these $A$-algebras and their 
underlying rings when $\gcd(i,j)=1$. Finally we obtain results that
fully determine when there are surjections onto $M_2(\field)$ when
$\field$ is a base field $\rats$ or $\ints_p$ for a prime number $p$.

\vspace{3 mm}


\vspace{2 mm}

\end{abstract}

\section{Introduction}
\label{sec:intro}

\subsection{History and motivation}
\label{subsec:hist}

Matrix rings and algebras have been studied for a long time. 
For examples of their importance and study see~\cite[Chapter 7]{lambook} 
and~\cite[Chapter 1, 6]{Rowen}. We say that a unital ring $R$ is a 
{\em complete $n\times n$ matrix ring over a ring $S$} if $R\cong M_n(S)$.
Recognizing a complete matrix ring, or algebra, is however not obvious. 
Recall that {\em $n\times n$ matrix units} in a unital ring $R$ is
a set of elements $\{ e_{i\/j} : 1\leq i,j\leq n\}\subseteq R$ that satisfy
\[
\sum_{i=1}^n e_{i\/i}=1_R \mbox{ and }
e_{i\/j}e_{k\/\ell}=\delta_{j\/k}e_{i\/\ell},
\]
where $\delta_{j\/\ell}$ is the Kronecker delta 
function~\cite[Chapter 1]{Rowen}.  
The most well-known element-wise characterization 
of a complete $n\times n$ matrix ring is given by the following 
folklore theorem~\cite[Prop.~11.3, p.~22]{Rowen}. 
\begin{theorem}
\label{thm:matrix-units}
A unital ring $R$ is a complete $n\times n$ matrix ring over some ring $S$,
that is $R\cong M_n(S)$, if and only if it contains a set of 
$n\times n$ matrix units.
\end{theorem}
The ring $S$ in the above Theorem~\ref{thm:matrix-units} is completely
determined by
\[
S = \left\{\sum_{i=1}^n e_{i\/1}xe_{1\/i} ; x\in R\right\}.
\]
However, these matrix units can be difficult to find and tedious to verify.
In 1990, Chatters 
in~\cite{quaternions} posed the following question: Let $\mathbb{H}$ be 
the integer quaternions and 
$T(n)=\begin{psmallmatrix}\mathbb{H} & n\mathbb{H}\\ 
\mathbb{H} & \mathbb{H} \end{psmallmatrix}$.  
For which, if any, values of $n$ is the tiled matrix 
ring $T(n)$ a complete matrix ring? At first glance, $T(n)$ does not 
appear to be a complete matrix ring. However, using properties of 
$\mathbb{H}$ and finding suitable matrix units, it turns out 
that $T(n)\cong M_2(S)$ for some $S$ (not necessarily unique) for 
odd values of $n$~\cite{Robson}.

In 1996, Agnarsson, Amitsur, and 
Robson in~\cite{Geira} refined structural results from~\cite{Robson}
and obtained the following two theorems, the first of which is a 
three-element relations.
\begin{theorem}[\cite{Geira}]
\label{thm:three}
A ring $R$ is a complete $(m+n)\times (m+n)$ matrix ring $M_{m+n}(S)$ 
if and only if it contains elements $a$, $b$, and $f$ satisfying the 
relations $af^m+f^nb=1\text{ and }f^{m+n}=0.$
\end{theorem}
Using this result they investigated rings of differential 
operators~\cite{Geira}.

In 1996, Lam and Leroy in~\cite{lampaper} investigated relations for 
recognizing matrix rings, in particular these three-element relations. 
Using the above Theorem~\ref{thm:three} from~\cite{Geira} they 
give an eigenring 
description, using a certain nilpotent element in $R$, for the ring 
$S$ over which $R$ is a complete $(m+n)\times (m+n)$ matrix ring.  In 
addition, they use Theorem~\ref{thm:three} to study Ore extension rings 
(or skew-polynomial rings).

Under these relations however, very little is known about the explicit 
structure  of the ring $S$.  In fact, under certain circumstances, 
$S$ may be the trivial ring.  Their next result is on two-element relations.
\begin{theorem}[\cite{Geira}]
\label{thm:two}
A ring $R$ is a complete $(m+n)\times (m+n)$ matrix ring $M_{m+n}(S)$ 
if and only if it contains elements $a$ and $f$ satisfying the relations 
$a^{m}f^{m}+f^na^n=1\text{ and }f^{m+n}=0.$
\end{theorem}
Note that the characterizations given in 
Theorem~\ref{thm:matrix-units} uses $n^2$ elements together
with $n^4 +1$ relations among them to characterize a complete 
$n\times n$ matrix ring, whereas Theorems~\ref{thm:three} and~\ref{thm:two}
use three and two elements respectively and two relations involving these
elements to characterize complete $(m+n)\times(m+n)$ matrix rings. 
In particular, the number or elements and relations in 
Theorems~\ref{thm:three} and~\ref{thm:two} are {\em not} functions
of the size of the matrix ring. These are the only known such 
characterizations for complete matrix rings.

Under the two-element relations in Theorem~\ref{thm:two}, 
it is easy to find matrices over $S$ that satisfy the two
relations: $a$ can be the matrix with 1's along its sub-diagonal 
and 0's everywhere else, while $f$ can be the matrix with 1's along its 
super-diagonal and 0's everywhere else. It is therefore natural
to ask what happens if the first relation in Theorem~\ref{thm:three}
is replaced by $a^if^m + f^na^j = 1$. The ring $R$ is by 
Theorem~\ref{thm:three} a complete $(m+n)\times(m+n)$ matrix ring,
but it could be the trivial ring; in~\cite{Geirb} it is shown that
if a ring $R$ contains elements $a$ and $b$ such that $ab^m+b^na = 1$
and $b^{m+n}=0$ where $m\neq n$, then $R$ is the trivial ring.
This result together with Theorem~\ref{thm:two} strongly suggest the 
study of the universal ring that contains two elements $a$ and $f$ 
that satisfy $a^if^m + f^na^j = 1$ and $f^{m+n} = 0$. This is the 
motivation for this article.

\subsection{Basic Setup and Definitions}
\label{subsec:defs}

The set $\{1,2,3,\ldots\}$ of natural numbers
will be denoted by $\nats$ and for $n\in\nats$ we let 
$[n] = \{1,2,\ldots,n\}$. The field of rational numbers will be
denoted by ${\rats}$ and for a prime $p$, the unique finite field with 
$p$ elements will be denoted 
by ${\ints_p}$. For the rest of this article, all rings 
will be associative and unital, that is with a multiplicative unit $1$,
and all homomorphisms will be assumed unital. We begin with some definitions.
\begin{definition}
Let $A$ be a commutative ring.

(I) The {\em{free monoid}} $\left\langle x_1,\dots,x_n\right\rangle$ 
on $n$ indeterminates is the set of words made
by the indeterminates $x_i$ along with the binary operation of 
concatenation of words with identity, the empty word, which we 
denote by $1$.

(II) The {\em{free $A$-algebra}} $A\left\langle x_1,\dots,x_n\right\rangle$ 
over the ring $A$ on $n$ indeterminates is the set of formal linear 
combinations over $A$ of elements 
from $\left\langle x_1,\dots,x_n\right\rangle$. Addition 
is defined as formal sums of elements and multiplication is defined as 
concatenation of basis elements extended as an $A$-bilinear operation.

(III) For a subset 
$\{f_1,\ldots,f_N\}\subseteq A\left\langle x_1,\dots,x_n\right\rangle$,
the {\em free $A$-algebra on $x_1,\ldots,x_n$ satisfying 
$f_1 = \cdots = f_n = 0$}, denoted by
$A\left\langle x_1,\ldots,x_n : f_1,\ldots, f_N\right\rangle$,
is the quotient $A$-algebra $A\left\langle x_1,\dots,x_n\right\rangle/I$
where $I = (f_1,\ldots,f_N)$ is the two-sided ideal of 
$A\left\langle x_1,\dots,x_n\right\rangle$ generated by $f_1,\ldots,f_N$.
\end{definition}
We are interested in rings with two elements, $x$ and $y$, satisfying the 
relations $x^iy^m+y^nx^j=1$ and $y^{m+n}=0$. In this article we will
investigate the free object satisfying these two relations.
\begin{definition}
\label{def:thething}
For a commutative ring $A$ and natural numbers $i,j,m,n\in\nats$, 
let $R(A;i,j,m,n) := A\langle x,y : x^iy^m+y^nx^j=1,y^{m+n}=0\rangle$.
\end{definition}
By Theorem~\ref{thm:three} we have $R(A;i,j,m,n)\cong M_{m+n}(S)$ for some 
$A$-algebra $S$. As we have very little information about $S$, 
we introduce the following sets similarly as in~\cite{Geirb}.
\begin{definition}
\label{def:ABC}
For a commutative ring $A$ define the sets
$\mathcal{A}_A, \mathcal{B}_A, \mathcal{C}_A\subseteq {\nats}^4$ 
as follows:

$\mathcal{A}_A$: the set of $(i,j,m,n)\in{\nats}^4$
with a non-trivial homomorphism $R(A;i,j,m,n) \rightarrow M_{m+n}(A)$.

$\mathcal{B}_A$: the set of $(i,j,m,n)\in{\nats}^4$
with a non-trivial homomorphism $R(A;i,j,m,n)\rightarrow M_N(A)$
for some $N\in {\nats}$.

$\mathcal{C}_A$: the set of $(i,j,m,n)\in{\nats}^4$ with
$R(A;i,j,n,m)$ non-trivial.
\end{definition}
{\sc Remark: }
Here $\mathcal{A}_A$ is the set of those $(i,j,m,n)\in{\nats}^4$ such 
that there exist $(m+n)\times(m+n)$ matrices $x$ and $y$ over $A$ satisfying
the defining relations for $R(A;i,j,m,n)$, $\mathcal{B}_A$ is the set of 
those $(i,j,m,n)\in{\nats}^4$ such that there exist some finite rank 
matrices $x$ and $y$ over $A$ satisfying the defining relations for 
$R(A;i,j,m,n)$, lastly, $\mathcal{C}_A$ is here the set of those 
$(i,j,m,n)\in{\nats}^4$ such that there exist ``infinite'' matrices 
$x$ and $y$ over $A$ satisfying the defining relations for $R(A;i,j,m,n)$.

A big motivating question is as follows:
\begin{question}
\label{qst:motivation}
For a given commutative ring $A$, can we describe each of the sets
$\mathcal{A}_A$, $\mathcal{B}_A$ and $\mathcal{C}_A$ from 
Definition~\ref{def:ABC}? 
\end{question}
The purpose of this article is in part to extend the results in~\cite{Geirb}
that partly answer the above Question~\ref{qst:motivation}. 
The article is organized as follows.

In Section~\ref{sec:m=n=1} we will derive numerous structural
properties of $R(A;i,j,1,1)\cong M_2(S)$ and completely describe
the ring $S$ for relatively prime integers $i$ and $j$.

In Section~\ref{sec:Ak-fields} we investigate $\mathcal{A}_{\field}$ when 
when $m = n = 1$ and 
$\field$ is one of the base fields $\rats$ or $\ints_p$ where $p$ is a prime
number. We will determine all $i$ and $j$ such that 
$(i,j,1,1)\in\mathcal{A}_{\field}$ for these base fields, that is, we 
determine when exactly there are $2\times 2$ matrices over 
$\field\in\{\rats,\ints_p\}$ that satisfy the defining relations
for $R(\field;i,j,1,1)$ from Definition~\ref{def:thething}.

\section{The case $m=n=1$.}
\label{sec:m=n=1} 

In this section, we will give an explicit description
of $R(A;i,j,m,n)$ from Definition~\ref{def:thething}
when $\gcd(i,j)=1$ and $m=n=1$.  

\subsection{Relations and Reductions}
\label{subsec:rel-red}

In this subsection we will derive some technical results for $R(A;i,j,1,1)$.
\begin{lemma}
\label{lmm:identity}
For the generators $x,y\in R(A;i,j,1,1)$, we have the following relations:
\[
yx^iy=yx^jy=y, \ \ yx^{i+j}y=0, \ \ yx^{2i}y=-yx^{2j}y.
\]
\end{lemma}
\begin{proof}
We have that $x^iy+yx^j=1$ and $y^2=0$ and hence
\[
y  = y\cdot 1 = y(x^iy+yx^j) =  yx^iy+y^2x^j =  yx^iy.
\]
Similarly we obtain $y=(x^iy+yx^j)y=yx^jy$.

Using the above we get
\[
yx^{i+j}y = 
yx^jx^iy = yx^j(x^iy+yx^j)x^iy = yx^{i+j}yx^iy+yx^jyx^{i+j}y = 2yx^{i+j}y,
\]
and so $yx^{i+j}y=0$.

Finally, using the last two results we get
\[
0 =  yx^{i+j}y  = yx^ix^jy = yx^i(x^iy+yx^j)x^jy
= yx^{2i}yx^jy+x^iyx^{2j}y = yx^{2i}y+yx^{2j}y,
\]
and thus $yx^{2i}y=-yx^{2j}y$.
\end{proof}
We can now prove our first theorem.
\begin{theorem}
\label{thm:ijji}
$R(k;i,j,1,1)=R(k;j,i,1,1)$
\end{theorem}
\begin{proof}
Using Lemma~\ref{lmm:identity}, $x^iy+yx^j=1$ and $y^2=0$, we get 
\begin{eqnarray*}
x^jy+yx^i & = & (x^iy+yx^j)(x^jy+yx^i)\\
& = & x^i(yx^jy)+x^iy^2x^j+yx^{2j}y+{yx^jy}x^i\\
& = & x^iy+yx^{2j}y+yx^i,
\end{eqnarray*}
and so $yx^{2j}y=x^jy-x^iy$.  Similarly, expanding $(x^jy+yx^i)(x^iy+yx^j)$, 
we get $yx^{2i}y=yx^i-yx^j$.  By Lemma~\ref{lmm:identity} again, we know that 
$yx^{2i}y=-yx^{2j}y$, so $x^jy+yx^i=x^iy+yx^j=1$, thus 
$R(A;i,j,1,1)\subseteq R(A;j,i,1,1)$.  By symmetry, we see $x^jy+yx^i=1$ 
and $y^2=0$ implies $x^iy+yx^j=1$ in $R(A;j,i,1,1)$ and so $x^iy+yx^j=1$,
thus $R(A;j,i,1,1)\subseteq R(A;i,j,1,1)$, and hence
$R(A;i,j,1,1)=R(A;j,i,1,1)$. 
\end{proof}
Note that Theorem~\ref{thm:ijji} is stronger than it appears;
it is clear that $R(A;i,j,1,1)$ is 
anti-isomorphic to  $R(A;j,i,1,1)$.  However, Theorem~\ref{thm:ijji} 
states the rings are actually equal as sets.
Without loss of generality, we can therefore assume either $i\leq j$
or $j\leq i$. For the rest of this section we will assume the latter.
\begin{lemma}
\label{lmm:injn}
For an arbitrary $n\in\nats$ and the generators $x,y\in R(A;i,j,1,1)$ 
we have:
\begin{align}
yx^{in}&=(-1)^nx^{jn}y+\sum_{k=0}^{n-1}(-1)^{n-1-k}x^{(n-1)j+k(i-j)} \label{eqn:a} \\
yx^{jn}&=(-1)^nx^{in}y+\sum_{k=0}^{n-1}(-1)^{k}x^{(n-1)j+k(i-j)} \label{eqn:b}
\end{align}
\end{lemma}
\begin{proof}
Since $yx^i=1-x^jy$ and $yx^j=1-x^iy$ by Theorem~\ref{thm:ijji}, 
we clearly have (\ref{eqn:a}) 
and (\ref{eqn:b}) for $n=1$. We proceed by induction on $n$ and
assume (\ref{eqn:a}) to hold for $n$. In that case we get
\begin{eqnarray*}
yx^{i(n+1)} & =  & \left(yx^{in}\right)x^i\\
  & = & (-1)^{n}x^{jn}yx^i + \sum_{k=0}^{n-1}(-1)^{n-1-k}x^{(n-1)j+k(i-j)+i}\\
  & = & (-1)^nx^{jn}(1-x^jy)+\sum_{k=0}^{n-1}(-1)^{n-(k+1)}x^{nj+(k+1)(i-j)}\\
  & = & (-1)^nx^{jn}+(-1)^{n+1}x^{j(n+1)}y+\sum_{k=1}^n(-1)^{n-k}x^{nj+k(i-j)}\\
  & = & (-1)^{n+1}x^{j(n+1)}y+\sum_{k=0}^{n}(-1)^{n-k}x^{nj+k(i-j)}.
\end{eqnarray*}
Thus, for every $n\in\nats$ we have (\ref{eqn:a}).
In exactly the same way we can prove (\ref{eqn:b}) by using
$yx^i = 1 - x^jy$ in the inductive step.
Hence, by induction we have (\ref{eqn:a}) and (\ref{eqn:b}) for 
every $n\in\nats$.
\end{proof}
\begin{lemma}
\label{lmm:polyy}
Let $p(x), q(x)\in A[x]$. If $p(x)y=q(x)y$ in $R(A;i,j,1,1)$, 
then $p(x)=q(x)$.
\end{lemma}
\begin{proof}
Suppose $p(x)y=q(x)y$ for some polynomials $p(x), q(x)\in A[x]$.
Then 
\[
p(x)yx^j = p(x)(1-x^iy) = p(x)-p(x)x^iy = p(x)-x^ip(x)y,
\]
since $p(x)$ is a polynomial in $x$ and thus commutes with $x^i$.  
By the same argument we have $q(x)yx^j=q(x)-x^iq(x)y$. 
Since $p(x)y=q(x)y$, we get
\[
p(x)-x^ip(x)y = p(x)yx^j = q(x)yx^j = q(x)-x^iq(x)y = q(x)-x^ip(x)y
\]
and hence $p(x)=q(x)$.
\end{proof}
The next lemma will come in handy later on.
\begin{lemma}
\label{lmm:invert}
If $i\neq j$, then $x$ is invertible in $R(A;i,j,1,1)$.
\end{lemma}
\begin{proof}
By Theorem~\ref{thm:ijji} we may without loss of generality assume $i>j$. 
We then get
\begin{eqnarray*}
1 & = & x^iy+yx^j\\
& = & x^{i-j}(x^jy)+yx^j\\
& = & x^{i-j}(1-yx^i)+yx^j\\
& = & (x^{i-j-1}-x^{i-j}yx^{i-1}+yx^{j-1})x
\end{eqnarray*}
Similarly, we can show $1=x(x^{j-1}y+x^{i-j-1}-x^{i-1}yx^{i-j})$,  
and so $x$ is invertible. 
\end{proof}
We are now able to show a useful relation for $x$.
\begin{theorem}
\label{thm:cyclic}
If $i > j$ and $\gcd(i,j)=d$, then in $R(A;i,j,1,1)$ we have 
\[
x^{((i+j)/d-1)(i-j)} = \sum_{k=1}^{(i+j)/d-1}(-1)^{k+1}x^{((i+j)/d-1-k)(i-j)}.
\]
\end{theorem}
\begin{proof}
We will evaluate the element $yx^{(ij)/d}y$ in two ways using our relations 
for $yx^{in}$ and $yx^{jn}$.  Here, let $a=i/d$ and $b=j/d$. By 
Lemma~\ref{lmm:injn} we have for $n=b$ and $n=a$ respectively that
\[
{yx^{(ij)/d}y=\sum_{k=0}^{b-1}(-1)^{b-1-k}x^{(b-1)j+k(i-j)}y} 
\mbox{ and } {yx^{(ij)/d}y=\sum_{k=0}^{a-1}(-1)^{k}x^{(a-1)j+k(i-j)}y},
\] 
and hence
\[
\sum_{k=0}^{a-1}(-1)^{k}x^{(a-1)j+k(i-j)}y
=\sum_{k=0}^{b-1}(-1)^{b-1-k}x^{(b-1)j+k(i-j)}y.
\]
Using Lemma~\ref{lmm:polyy} we get
\[
{0=\sum_{k=0}^{b-1}(-1)^{b-k}x^{(b-1)j+k(i-j)}
+\sum_{k=0}^{a-1}(-1)^{k}x^{(a-1)j+k(i-j)}}.
\]
Since $aj=bi$ we have, $(b-1)j+k(i-j)=(a-1)j+(k-b)(i-j)$, and so 
\[
\sum_{k=0}^{a-1}(-1)^kx^{(a-1)j+k(i-j)}
=\sum_{k=b}^{a+b-1}(-1)^{b-k}x^{(b-1)j+k(i-j)},
\]
and therefore 
\begin{eqnarray*}
 0 & = &\sum_{k=0}^{b-1}(-1)^{b-k}x^{(b-1)j+k(i-j)}
 +\sum_{k=b}^{a+b-1}(-1)^{b-k}x^{(b-1)j+k(i-j)} \\
   & = &\sum_{k=0}^{a+b-1}(-1)^{b-k}x^{(b-1)j+k(i-j)}.
\end{eqnarray*}  
Since $x$ is invertible by Theorem~\ref{lmm:invert}, we obtain 
\[
0=\sum_{k=0}^{a+b-1}(-1)^{b-k}x^{k(i-j)},
\]
and by re-indexing and shifting the last term we have, 
\[
x^{(a+b-1)(i-j)}=\sum_{k=1}^{a+b-1}(-1)^{k+1}x^{(a+b-1-k)(i-j)}.
\]
\end{proof}
Writing out the above sum, we see that the relation in Theorem~\ref{thm:cyclic} 
gives an alternating series relation for $x^{(i+j)/(d-1)(i-j)}$.  Letting 
$m=(i+j)/d$ we have 
\[
x^{(m-1)(i-j)}=x^{(m-2)(i-j)}-x^{(m-3)(i-j)}+\dots+(-1)^{(i+j)/d}.
\]
Theorem~\ref{thm:cyclic} can be used to obtain the following.
\begin{corollary}
\label{cor:rootofunity}
If $d=\gcd(i,j)$, then $x^{(i^2-j^2)/d}=(-1)^{(i+j)/d}$.
\end{corollary}
\begin{proof}
By Theorem~\ref{thm:cyclic}, we have
$x^{((i+j)/d-1)(i-j)} =\sum_{k=1}^{(i+j)/d-1}(-1)^{k+1}x^{((i+j)/d-1-k)(i-j)}$.  
Multiplying both 
sides of this relation by $x^{i-j}$ and using Theorem~\ref{thm:cyclic} again, 
we get
\begin{eqnarray*}
{x^{{(i+j)(i-j)}/{d}}} 
& = & {\sum_{k=1}^{(i+j)/d-1}(-1)^{k+1}x^{((i+j)/d-k)(i-j)}}\\
& = & x^{((i+j)/d-1)(i-j)}+{\sum_{k=2}^{(i+j)/d-1}(-1)^{k+1}x^{((i+j)/d-k)(i-j)}}\\
& = & {\sum_{k=1}^{(i+j)/d-1}(-1)^{k+1}x^{((i+j)/d-k-1)(i-j)}
  -\sum_{k=1}^{(i+j)/d-2}(-1)^{k+1}x^{((i+j)/d-k-1)(i-j)}}\\
& = & (-1)^{(i+j)/d}.
\end{eqnarray*}
\end{proof}
If $i\neq j$, Corollary~\ref{cor:rootofunity} shows that $x$ can be viewed
as a root of unity.
\begin{proposition}
\label{prp:centerR}
The elements $x^{i+j}$ and $x^i-x^j$ are in the center of $R(A;i,j,1,1)$.
\end{proposition}
\begin{proof}
It suffices to show that $x^{i+j}$ and $x^i-x^j$ commute with $y$. By  
definition of $R(A;i,j,1,1)$ we get
\[
y(x^{i+j}) = (yx^i)x^j = (1-x^jy)x^j 
= x^j-x^jyx^j = x^j-x^j(1-x^iy) = (x^{i+j})y,
\]
and 
\[
y(x^i-x^j) = yx^i-yx^j = (1-x^jy)-(1-x^iy) = (x^i-x^j)y.
\]
Hence, $x^{i+j}$ and $x^i-x^j$ commute with both $x$ and $y$ and therefore
with each element in $R(A;i,j,1,1)$.
\end{proof}
\begin{lemma}
\label{lmm:yxd}
If $\gcd(i,j)=d$, then there exists polynomials $p(x),q(x)\in A[x]$, 
both alternating sums of powers of $x$, such that 
$yx^d=p(x)+q(x)y$ in $R(A;i,j,1,1)$.
\end{lemma}
\begin{proof}
If $i=j$, then $\gcd(i,j)=i$ and we have $yx^i=1-x^iy$.

Now suppose $i\neq j$ and let $d=\gcd(i,j)$.  Since $d=\gcd(i+j,j)$, 
there exist $m,n\in{\nats}$ such that $d=nj-m(i+j)$,  and so 
$yx^{nj}=yx^{m(i+j)+d}$.  By Proposition~\ref{prp:centerR}, $x^{i+j}$ is in the 
center of $R(A;i,j,1,1)$ and so by Lemma~\ref{lmm:injn} we obtain 
\begin{align}
\label{eqn:c}
x^{m(i+j)}yx^d=yx^{m(i+j)+d}=yx^{nj}=(-1)^nx^{in}y
+\sum_{k=0}^{n-1}(-1)^kx^{(n-1)j+k(i-j)}. 
\end{align}  
Now, by Corollary~\ref{cor:rootofunity}, 
the inverse of $x$ is a power of $x$ and 
hence $x^rx^{m(i+j)}=1$ for some $r$. Therefore, by (\ref{eqn:c}) we have 
\[
yx^d=(-1)^nx^{in+r}y+\sum_{k=0}^{n-1}(-1)^kx^{(n-1)j+k(i-j)+r}.
\] 
\end{proof}
\begin{theorem}
\label{thm:R=fin-gen}
If $i>j$ and $\gcd(i,j)=1$, then $R(A;i,j,1,1)$ is a finitely generated 
$A$-module with a generating set of cardinality at most $2(i+j-1)(i-j)$.
\end{theorem}
\begin{proof}
Since $\gcd(i,j)=1$, by Theorem~\ref{thm:cyclic} and Lemma~\ref{lmm:yxd}, 
we have 
relations which work as reductions for $x^n$ and $yx$ respectively, 
for $n\geq(i+j-1)(i-j)$. Using these reductions, we can write every 
monomial/word of $x$ and $y$ in $R(A;i,j,1,1)$ as an $A$-linear 
combination of elements from 
$\{1,x,x^2,\dots,x^{(i+j-1)(i-j)-1},y,xy,\dots, x^{(i+j-1)(i-j)-1}y\}$. 
Hence, $R(A;i,j,1,1)$ is a finitely-generated $A$-module with 
generating set of cardinality at most $2(i+j-1)(i-j)$.
\end{proof}

\subsection{Matrix Descriptions}
\label{subsec:descr}

In this subsection we will obtain a complete description of the 
$A$-algebra $R(A;i,j,1,1)$ when $\gcd(i,j)=1$. Letting $n=2$, $a=x^i$, 
$b=x^j$, and $f=y$, then $\{E_{hk} : 1\leq h,k\leq 2\}$, where 
$E_{hk}=y^{h-1}x^iyx^{j(k-1)}$, forms a set of $2\times 2$ matrix units
by Theorem~1.3 of~\cite{Geira} and so we have the following.
\begin{observation}
\label{obs:2x2}
There exists an $A$-algebra $L$ such that $R(A;i,j,1,1)\cong M_2(L)$.
\end{observation}
If we let $e_{hk}=E_{(3-h)(3-k)}$ for each $h$ and $k$, then 
it is easy to verify the $2^4 + 1 = 17$ relations from 
Theorem~\ref{thm:matrix-units} to show that $\{e_{hk}:1\leq h,k\leq 2\}$ 
also forms a complete set of $2\times 2$ 
matrix units, where $e_{11}=yx^j$, $e_{12}=y$, $e_{21}=x^iyx^j$, and 
$e_{22}=x^iy$. 

By Observation~\ref{obs:2x2} and the $2\times 2$ matrix units 
$\{e_{hk}:1\leq h,k\leq 2\}$, we have an isomorphism 
$\phi:R(A;i,j,1,1)\rightarrow M_2(L)$.  Identifying $R(A;i,j,1,1)$ 
with $M_2(L)$ via $\phi$, we have 
$yx^j= \begin{psmallmatrix} 1 &  0 \\ 0 & 0 \end{psmallmatrix}$
$y= \begin{psmallmatrix} 0 &  1 \\ 0 & 0 \end{psmallmatrix}$,
$x^iyx^j= \begin{psmallmatrix} 0 &  0 \\ 1 & 0 \end{psmallmatrix}$ and
$x^iy= \begin{psmallmatrix} 0 &  0 \\ 0 & 1 \end{psmallmatrix}$
since $yx^j=e_{11}$, $y=e_{12}$, $x^iyx^j=e_{21}$, and $x^iy=e_{22}$. 
Hence, we will now and for the rest of this
subsection view $R(A;i,j,1,1)$ as the matrix ring $M_2(L)$.
Letting 
$x^j= \begin{psmallmatrix} a &  b \\ c & d \end{psmallmatrix}$
and
$x^i= \begin{psmallmatrix} p &  q \\ s & t \end{psmallmatrix}$ 
for some $a,b,c,d,p,q,s,t\in L$, we get from $y\cdot x^i = yx^i$ 
and $x^i\cdot y = x^iy$ the following matrix equations
$\begin{psmallmatrix} 0 &  1 \\ 0 & 0 \end{psmallmatrix}
\begin{psmallmatrix} a &  b \\ c & d \end{psmallmatrix}
=\begin{psmallmatrix} 1 &  0 \\ 0 & 0 \end{psmallmatrix}$
and
$\begin{psmallmatrix} p & q \\ s & t \end{psmallmatrix}
\begin{psmallmatrix} 0 & 1 \\ 0 & 0 \end{psmallmatrix}
=\begin{psmallmatrix} 0 &  0 \\ 0 & 1 \end{psmallmatrix}$,
and so $d=p=0$ and $c=s=1$. Since $x^jx^i=x^ix^j$, we also obtain 
$\begin{psmallmatrix} a &  b \\ 1 & 0 \end{psmallmatrix}
\begin{psmallmatrix} 0 &  q \\ 1 & t \end{psmallmatrix}
=\begin{psmallmatrix} 0 &  q \\ 1 & t \end{psmallmatrix}
\begin{psmallmatrix} a &  b \\ 1 & 0 \end{psmallmatrix}$
and hence $q=b$, $t=-a$, $aq=-bt$, and so $ab=ba$. Further, 
\begin{equation}
\label{eqn:bs-at}
x^{i+j} = \begin{pmatrix} b &  0 \\ 0 & b \end{pmatrix}
\mbox{ and } 
x^j-x^i = \begin{pmatrix} a &  0 \\ 0 & a \end{pmatrix},
\end{equation}
and so by Proposition~\ref{prp:centerR} we have the following.
\begin{claim}
\label{clm:abcenter}
Here $a,b\in L$ are in the center of $L$ and thus 
$A[a,b]\subseteq L$.
\end{claim}
Suppose now $\gcd(i,j) = 1$. If $i > j$, then since $x$ is invertible 
by Theorem~\ref{lmm:invert}, then so is $x^{i+j}$ and also $b$
by Corollary~\ref{cor:rootofunity} in $A[a,b]$.
Further, there are $\alpha, \beta\in {\nats}_0$ such that 
$1=\alpha j-\beta i$ and so 
\[
x=x^{\alpha j-\beta i}
=\begin{pmatrix} a &  b \\ 1 & 0 \end{pmatrix}^\alpha
\begin{pmatrix} 0 &  b \\ 1 & -a \end{pmatrix}^{-\beta}
=\begin{pmatrix} a &  b \\ 1 & 0 \end{pmatrix}^\alpha
\begin{pmatrix} a &  b \\ 1 & 0 \end{pmatrix}^\beta\frac{1}{b^\beta}
\in M_2(A[a,b]).
\]
If $i=j$, then $i=j=1$ and hence  
$x=\begin{psmallmatrix} a &  b \\ 1 & 0 \end{psmallmatrix}
=\begin{psmallmatrix} 0 &  b \\ 1 & -a \end{psmallmatrix}$
and so $a=0$ and 
$x=\begin{psmallmatrix} 0 &  b \\ 1 & 0 \end{psmallmatrix}\in M_2(A[a,b])$.
Therefore, if $\gcd(i,j)=1$, then in the isomorphism 
$R(A;i,j,1,1)\cong M_2(L)$ we see that both $x$ and $y$ 
are mapped to $M_2(A[a,b])$ and so all of $R(A;i,j,1,1)$  is mapped
to $M_2(A[a,b])$. Therefore $M_2(L) = M_2(A[a,b])$ and 
so $L=A[a,b]$. This is summed up in the following proposition.
\begin{proposition}
\label{prp:L-comm}
If $a$ and $b$ are as in Claim~\ref{clm:abcenter}, $\gcd(i,j)=1$ and $i>j$ 
then $R(A;i,j,1,1)\cong M_2(A[a,b])$ as $A$-algebras. 
In particular, $L$ from Observation~\ref{obs:2x2} is commutative.
\end{proposition}
Note that to obtain Proposition~\ref{prp:L-comm} we
use the fact that the {\em equality} $M_2(L) = M_2(A[a,b])$
implies $L = A[a,b]$.
\comment{ 
Namely, we always have the following.
\begin{observation}
If $M_n(A)\cong M_n(B)$ for commutative rings (algebras) $A$ and $B$ 
where $n\geq 1$, then $A\cong B$ as rings (algebras).
\end{observation}
\begin{proof}
Looking at the centers of $M_n(A)$ and $M_n(B)$ we get
$A\cong Z\left(M_n(A)\right)\cong Z\left(M_n(B)\right)\cong B$.
\end{proof}
}
If the equality in $M_2(L) = M_2(A[a,b])$ is replaced by an isomorphism,
we cannot draw the same conclusions: In the papers~\cite{smith} 
and~\cite{nonisomorphic} examples of non-isomorphic, non-commutative 
rings $A$ and $B$ are given such that $M_n(A)\cong M_n(B)$.
In fact, it is shown in~\cite{nonisomorphic} that there is an uncountable 
family of pairwise non-isomorphic rings $\{S_\alpha\}$ such that 
$M_2(S_\alpha)\cong M_2(S_\beta)$. Further, all $S_\alpha$ 
are Noetherian domains that are finitely-generated over their centers.

The rest of this subsection is devoted to the description 
of the commutative $A$-algebra $L = A[a,b]$ when we have 
$\gcd(i,j)=1$.  Under this 
assumption we get by Lemma~\ref{lmm:yxd} a 
commuting reduction rule for $yx$ in 
$R(A;i,j,1,1)$.  For this analysis we begin with a definition.
\begin{definition}
\label{def:frec}
Let $A[s,t]$ be the polynomial ring in two variables $s$ and $t$ over $A$ 
and let $f:{\nats}_0\rightarrow A[s,t]$ be defined recursively in the 
following way: $f(0)=0$, $f(1)=1$, and $f(n)=tf(n-1)+sf(n-2)$ for $n\geq 2$.
\end{definition}
The following is easily obtained by induction on $n$.
\begin{lemma}
\label{lmm:matrix}
For $n\geq 1$, we have 
$\begin{psmallmatrix} t & s\\ 1 & 0 \end{psmallmatrix}^n
=\begin{psmallmatrix} f(n+1) & sf(n)\\ f(n) & sf(n-1) \end{psmallmatrix}$.
\end{lemma}
The main result of this section is the following theorem,
in which a complete description of the algebra $A[a,b]$ from
Proposition~\ref{prp:L-comm} is given.
\begin{theorem}
\label{thm:1st-main}
Let $f$ be as in Definition~\ref{def:frec}. If $\gcd(i,j)=1$, then 
$R(A;i,j,1,1)\cong M_2(A[s,t]/I)$ where 
\[
I=\left( f(i+j), f(i+j-1)-s^{j-1}, s^{i-j}-(-1)^{i-j}\right)
\]
and $A[s,t]$ is the polynomial ring in two indeterminates $s$ and $t$ over $A$.
\end{theorem}
Before beginning our proof, we discuss some interesting consequences of 
Proposition~\ref{prp:L-comm} and Theorem~\ref{thm:1st-main}.

First, for any ring $R$, the matrix ring $M_n(R)$ and $R$ are 
{\em Morita equivalent} (see~\cite{Morita}), meaning there is an 
equivalence of their modules in a categorical sense.  
By Theorem~\ref{thm:1st-main} we have in particular that 
{\em $R(A;i,j,1,1)$ is Morita equivalent to a commutative ring when 
$\gcd(i,j)=1$.}

For the second consequence, we recall that
a {\em polynomial identity ring} (or a {\em PI ring} for short)
$R$ is a ring such that 
there exists a polynomial in non-commuting indeterminates 
$p(x_1,x_2,\dots,x_n)\in R\langle x_1,x_2,\dots,x_n\rangle$ such that 
$p(r_1,r_2,\dots,r_n)=0$ for all $r_i\in R$. For example, any commutative 
ring $R$ is a PI ring since it satisfies 
the identity $xy-yx=0$. Similarly, any $2\times 2$ matrix ring over 
a commutative ring is also 
a PI ring since it satisfies the {\em Hall identity} $(xy-yx)^2z=z(xy-yx)^2$.
In fact, any $n\times n$ matrix ring over a commutative ring satisfies
the {\em Amitsur-Levitzki identity}~\cite{Rowen}  
\[
S_{2n}(x_1,x_2,\dots,x_{2n})
=\sum_{\pi\in\text{Sym}(2n)}\text{sgn}(\pi)x_{\pi(1)}x_{\pi(2)}\dots x_{\pi(2n)}=0.
\]
Therefore, by Theorem~\ref{thm:1st-main} $R(A;i,j,1,1)$ satisfies 
both the Hall identity and the Amitsur-Levitzki identity $S_4(x_1,x_2,x_3,x_4)$
and is therefore a PI ring when $\gcd(i,j)=1$.

We now delve into the proof of Theorem~\ref{thm:1st-main}.
\begin{proof}
[Theorem~\ref{thm:1st-main}] 
{\sc First case:} Suppose $i=j$, and so $i,j=1$. This case
is special, in great part since $L$ from Observation~\ref{obs:2x2} is here not
a finitely generated $A$-module, unlike the case when
$i>j$ (see Theorem~\ref{thm:R=fin-gen}.)

First we note that in this case $I=(f(2),f(1)-s^0,s^0-(-1)^0)=(t)$
and so $A[s,t]/I = A[s,t]/(t) = A[s]$, the polynomial ring over
$A$ in one indeterminate $s$.

In the $A$-algebra $R(A;1,1,1,1)$ we can show that 
$\{yx=1-xy, y^2=0\}$ forms a complete set of reductions 
under the {\em degree lexicographic order}, or {\em deglex order}
for short w.r.t. $x < y$ and hence $R(A;1,1,1,1)$ 
has a free $A$-module basis given by $\{1,x,x^2,\dots,y,xy,x^2y,\dots\}$.
To see that $R(A;1,1,1,1)$ is isomorphic to the $2\times 2$ matrix
algebra over $A[s]$, we introduce the variable $s:=x^2\in R(A;1,1,1,1)$
and so $R(A;1,1,1,1) = k\langle x,y,s : yx=1-xy, y^2=0, x^2=s\rangle$.  
This yields the following set of reductions 
$\{yx=1-xy, y^2=0, x^2=s\}$ under the deglex order $s < x < y$, 
which is not complete as in the Diamond Lemma by Bergman 
from~\cite{diamond}, since it contains two 
{\em overlap ambiguities} $yx^2$ and $x^3$. Resolving the first one we get
\[
yx^2 = (yx)x = (1-xy)x = x - xyx = x - x(1-xy) = x^2y = sy,
\]
on one hand, and $yx^2 = y(x^2) = ys$ on the other, and hence we
obtain a new relation $sy = ys$. Resolving the second one
we get $x^3 = x(x^2) = xs$ on one hand, and $x^3 = (x^2)x = sx$ on 
the other, and hence we obtain a new relation $sx = xs$. By adding
these two new relations to our system of deglex reductions we 
obtain a complete set of reductions $\{yx=1-xy, y^2=0, x^2=s, xs=sx, ys=sy\}$
under the deglex order. Therefore $s$ is in the center of $R(A;1,1,1,1)$ 
and so $R(A;1,1,1,1)$ is an $A[s]$-algebra.
By our complete set of reductions we see that $R(A;1,1,1,1)$ is 
a free $A[s]$-module with a basis consisting of $\{1,x,y,xy\}$ and hence
of rank $4$ over $A[s]$. Further, $R(A;1,1,1,1)$ has the set of 
$2\times 2$ matrix units $e_{11}=1-xy,\, e_{12}=y,\,e_{21}=x-sy,\, e_{22}=xy$,
which shows $R(A;1,1,1,1)\cong M_2(L)$ for some $A[s]$-algebra $L$.  
Since both $R(A;1,1,1,1)$ and $M_2(L)$ have rank $4$ over $A[s]$, we
have $T=A[s]$ and hence $R(A;1,1,1,1)\cong M_2(A[s])$ as $A[s]$-algebras.

{\sc Second case:} Suppose $i>j$. In this case
$L$ is a finitely generated $A$-module by Theorem~\ref{thm:R=fin-gen}.
Let $I=(f(i+j), f(i+j-1)-s^{j-1},s^{i-j}-(-1)^{i-j})$ be as stated in 
Theorem~\ref{thm:1st-main}.
We note that $s^{-1}$ exists in $A[s,t]/I$ and is given by 
$s^{-1}=(-1)^{i-j}s^{i-j-1}.$ Since $\gcd(i,j)=1$, there are  
$\alpha,\beta\in {\nats}_0$ such that $\alpha j-\beta i=1$. Let 
$X,Y \in M_2(A[s,t]/I)$ be given by 
$X=\frac{1}{s^{\beta}}\begin{psmallmatrix} t & s\\ 1 & 0 
\end{psmallmatrix}^{\alpha+\beta}$ and
$Y=\begin{psmallmatrix} 0 & 1\\ 0 & 0 \end{psmallmatrix}$.
By definition of $I$ we have
in $A[s,t]/I$ that $sf(i+j) = f(i+j) = 0$, $sf(i+j-1) = s\cdot s^{j-1} = s^j$
and $f(i+j+1) = tf(i+j) + sf(i+j-1) = s^j$ and hence in $M_2(A[s,t]/I)$
we have by Lemma~\ref{lmm:matrix} that 
\begin{equation}
\label{eqn:ts-j}
\begin{pmatrix}
t & s\\
1 & 0
\end{pmatrix}^{i+j}=
\begin{pmatrix}
s^j & 0 \\
0   & s^j
\end{pmatrix}
\end{equation}
and hence, since $\alpha j - \beta i = 1$ we have 
$(\alpha + \beta)j = 1 + \beta(i+j)$ and so we get from 
(\ref{eqn:ts-j}) that 
\[
X^j = 
\frac{1}{s^{\beta j}}
\begin{pmatrix}
t & s\\
1 & 0
\end{pmatrix}^{(\alpha +\beta)j} =  
\frac{1}{s^{\beta j}}
\left[
\begin{pmatrix}
t & s\\
1 & 0
\end{pmatrix}^{i+j}
\right]^{\beta} 
\begin{pmatrix}
t & s\\
1 & 0
\end{pmatrix} =  
\begin{pmatrix}
t & s\\
1 & 0
\end{pmatrix}.
\]
Similarly, since
$\begin{psmallmatrix} t & s\\ 1 & 0 \end{psmallmatrix}^{-1} = 
\frac{1}{s}\begin{psmallmatrix} 0 & s\\ 1 & -t \end{psmallmatrix}$ and
$(\alpha+\beta)i = \alpha(i+j) - 1$ we get
\[
X^i = 
\frac{1}{s^{\beta i}}
\begin{pmatrix}
t & s\\
1 & 0
\end{pmatrix}^{(\alpha +\beta)i} =  
\frac{1}{s^{\beta i}}
\begin{pmatrix}
s^{\alpha j} & 0\\
0 & s^{\alpha j}
\end{pmatrix}
\frac{1}{s}
\begin{pmatrix}
0 & s\\
1 & -t
\end{pmatrix} =
\begin{pmatrix}
0 & s\\
1 & -t
\end{pmatrix}.
\]
Therefore we have $X^iY+YX^j=I$ and $Y^2=0$ in $M_2(A[s,t]/I)$ and 
hence there is a well-defined $A$-algebra homomorphism 
$R(A;i,j,1,1) \rightarrow M_2(A[s,t]/I)$ with $x\mapsto X$ and 
$y\mapsto Y$. For this map we further have 
$x^{i+j}\mapsto \begin{psmallmatrix} s & 0 \\0 & s \end{psmallmatrix}$ and
$x^j-x^i\mapsto \begin{psmallmatrix} t & 0\\ 0 & t \end{psmallmatrix}$,
and so this homomorphism is a surjection. By (\ref{eqn:bs-at}) this
homomorphism induces an $A$-algebra surjection 
$L = A[a,b]\twoheadrightarrow A[s,t]/I$ where $b\mapsto s$ and $a\mapsto t$ 
since $L$ is commutative. 
This means, in particular, that $t$ and $s$ satisfy any equations
that $a$ and $b$ do in $L = A[a,b]$. It remains to show that $a$ and $b$
satisfy the same relations over $A$ that $s$ and $t$ do in $A[s,t]/I$.

Since Lemma~\ref{lmm:matrix} holds for arbitrary $s$ and $t$, then
we get the same equations by replacing $s$ with $a$ and $t$ with $b$ and
hence we have in $R(A;i,j,1,1)\cong M_2(A[a,b])$ on one hand that 
$(x^j)^{i+j}
=\begin{psmallmatrix}
a & b\\
1 & 0
\end{psmallmatrix}^{i+j}
=\begin{psmallmatrix}
f(i+j+1) & bf(i+j)\\
f(i+j) & bf(i+j-1)
\end{psmallmatrix}$ 
and on the other hand we obtain
$(x^j)^{i+j}=(x^{i+j})^j =\begin{psmallmatrix} b & 0\\ 0 & b 
\end{psmallmatrix}^{j}
=\begin{psmallmatrix} b^j & 0\\ 0 & b^j \end{psmallmatrix}$.
Since by Lemma~\ref{lmm:invert} $x$ is invertible, then so is $b$ in $L$,
and we then obtain from the above two equations that $f(i+j-1)=b^{j-1}$ 
and $f(i+j)=0$. We therefore get that $a$ and $b$ satisfy the same 
relations in $L = A[a,b]$ as $t$ and $s$ do in 
$I$ respectively and so $L\cong A[s,t]/I$ and hence 
$R(A;i,j,1,1)\cong M_2(A[s,t]/I)$ which completes our proof.
\end{proof}

In~\cite{Geirb}, it is shown that for $A=k$ a field, $R(k;i,j,1,1)$ 
always maps to some $M_N(k)$ and is therefore non-zero. It remains to show 
that if $\gcd(i,j)=1$, then for any commutative ring $A$ we have
$R(A;i,j,1,1)\neq\{0\}$.  
For that, we need a few technical results for the function $f$. The 
following lemma can be shown with simple induction arguments.
\begin{lemma}
\label{lmm:tpoly}
As a polynomial of $t$, we have
(I) $f(n)$ is monic with degree $n-1$,
(II) $f(2n)$ has no constant term,
(III) $f(2n+1)$ has constant term $s^n$.
\end{lemma}
\begin{lemma}
\label{lmm:technical}
Let $\bar{f}$ be the image of $f$ under the map $A[s,t]\to A[t]$,
$s\mapsto -1$, so $\bar{f}(n)=t\bar{f}(n-1)-\bar{f}(n-2)$.  In 
this case we have for each $n\geq 1$: 
\begin{eqnarray}
\bar{f}(2n-1) &  =& (\bar{f}(n)+\bar{f}(n-1))(\bar{f}(n)-\bar{f}(n-1)), 
\label{eqn:aa}\\
\bar{f}(2n)-1 & = & (\bar{f}(n+1)-\bar{f}(n))(\bar{f}(n)+\bar{f}(n-1)),
\label{eqn:bb}\\
\bar{f}(2n)+1 & = & (\bar{f}(n+1)+\bar{f}(n))(\bar{f}(n)-\bar{f}(n-1)).
\label{eqn:cc}
\end{eqnarray}
\end{lemma}
\begin{proof}
We first note that (\ref{eqn:aa}) is by Definition~\ref{def:frec} clearly
true for $n=1,2$. We proceed by induction on $n$. Suppose now
\[
\bar{f}(2m-1)=(\bar{f}(m)+\bar{f}(m-1))(\bar{f}(m)-\bar{f}(m-1))
=\bar{f}(m)^2-\bar{f}(m-1)^2
\]
for all $m\leq n$.  Then, by the defining recursion, we get
\begin{eqnarray*}
\bar{f}(2n+1) & = & t\bar{f}(2n)-\bar{f}(2n-1)\\
  & = & t[t\bar{f}(2n-1)-\bar{f}(2n-2)]-\bar{f}(2n-1)\\
  & = & t^2\bar{f}(2n-1)-t\bar{f}(2n-2)-\bar{f}(2n-1)\\
  & = & t^2\bar{f}(2n-1)-[\bar{f}(2n-1)+\bar{f}(2n-3)]-\bar{f}(2n-1)\\
  & = & (t^2-2)\bar{f}(2n-1)-\bar{f}(2n-3).
\end{eqnarray*}
Using the induction hypothesis for $\bar{f}(2n-1)$ and
$\bar{f}(2n-3)$ and the defining recursion, we further get
\begin{eqnarray*}
\bar{f}(2n+1) & = & 
  (t^2-2)[\bar{f}(n)^2-\bar{f}(n-1)^2]-[\bar{f}(n-1)^2-\bar{f}(n-2)^2]\\
  & = & (t^2-2)[(t\bar{f}(n-1)-\bar{f}(n-2))^2-\bar{f}(n-1)^2]
-[\bar{f}(n-1)^2-\bar{f}(n-2)^2]\\
  & = & (t^4-3t^2+1)\bar{f}(n-1)^2-(2t^3-4t)\bar{f}(n-1)\bar{f}(n-2)\\
  &   & + (t^2-1)\bar{f}(n-2)^2.
\end{eqnarray*}
Again, using the defining recurrence for $\bar{f}(n+1)$ and $\bar{f}(n)$, 
we obtain
\begin{eqnarray*}
\bar{f}(n+1)^2-\bar{f}(n)^2 
  & = & (\bar{f}(n+1)+\bar{f}(n))(\bar{f}(n+1)-\bar{f}(n))\\
  & = & \left([t\bar{f}(n)-\bar{f}(n-1)]+\bar{f}(n)\right)
    \left([t\bar{f}(n)-\bar{f}(n-1)]-\bar{f}(n)\right)\\
  & = & ((t+1)\bar{f}(n)-\bar{f}(n-1))((t-1)\bar{f}(n)-\bar{f}(n-1))\\
  & = & ((t+1)[t\bar{f}(n-1)-\bar{f}(n-2)]-\bar{f}(n-1))\\
  &   & \cdot((t-1)[t\bar{f}(n-1)-\bar{f}(n-2)]-\bar{f}(n-1))\\
  & = & ((t^2+t-1)\bar{f}(n-1)-(t+1)\bar{f}(n-2))\\
  &   & \cdot((t^2-t-1)\bar{f}(n-1)-(t-1)\bar{f}(n-2))\\
  & = & (t^4-3t^2+1)\bar{f}(n-1)^2-(2t^3-4t)\bar{f}(n-1)\bar{f}(n-2)\\
  &   & +(t^2-1)\bar{f}(n-2)^2.
\end{eqnarray*}
Hence, we obtain from the last two displayed relations 
\[
\bar{f}(2n+1)=\bar{f}(n+1)^2-\bar{f}(n)^2 =
(\bar{f}(n+1)+\bar{f}(n-1))(\bar{f}(n+1)-\bar{f}(n)),
\]
and thus (\ref{eqn:aa}) is proved by induction.

We will use induction to prove both (\ref{eqn:bb}) and (\ref{eqn:cc}) 
simultaneously. For $n=1$ we have $\bar{f}(2) = t$, $\bar{f}(1) = 1$ and 
$\bar{f}(0) = 0$ and so 
$\bar{f}(2)-1=t-1=(\bar{f}(2)-\bar{f}(1))(\bar{f}(1)+\bar{f}(0))$ and 
$\bar{f}(2)+1=t+1=(\bar{f}(2)+\bar{f}(1))(\bar{f}(1)-\bar{f}(0))$.
Suppose 
$\bar{f}(2m)-1=(\bar{f}(m+1)-\bar{f}(m))(\bar{f}(m)+\bar{f}(m-1))$ and 
$\bar{f}(2m)+1=(\bar{f}(m+1)+\bar{f}(m))(\bar{f}(m)-\bar{f}(m-1))$
for all $m\leq n$.  Using the defining recurrence, (\ref{eqn:aa}), and our 
induction hypothesis we get,
\begin{eqnarray*}
\bar{f}(2n+2) & =  & t\bar{f}(2n+1) - \bar{f}(2n)\\
  & =  & t(\bar{f}(n+1)-\bar{f}(n))(\bar{f}(n+1)+\bar{f}(n))
-\left[(\bar{f}(n+1)-\bar{f}(n))(\bar{f}(n)+\bar{f}(n-1))+1\right]\\
  & =  & (\bar{f}(n+1)-\bar{f}(n))(t\bar{f}(n+1)+t\bar{f}(n)
-\bar{f}(n)-\bar{f}(n-1))-1\\
  & =  & (\bar{f}(n+1)-\bar{f}(n))(\bar{f}(n+2)+\bar{f}(n+1))-1.
\end{eqnarray*}
Thus $\bar{f}(2n+2)+1=(\bar{f}(n+1)-\bar{f}(n))(\bar{f}(n+2)+\bar{f}(n+1))$.  
Similarly, 
$\bar{f}(2n+2)-1=(\bar{f}(n+1)+\bar{f}(n))(\bar{f}(n+2)-\bar{f}(n+1))$, 
which completes our proof.
\end{proof}
We now argue directly that if $\gcd(i,j)=1$, then 
$R(A;i,j,1,1)\neq \{0\}$ for any commutative ring $A$.
\begin{theorem}
\label{thm:notzero}
If $\gcd(i,j)=1$ and $I$ is as in Theorem~\ref{thm:1st-main},
then $I\neq A[s,t]$ and thus $R\neq \{0\}$.
\end{theorem}
\begin{proof}
{\sc First case:} Suppose $i+j$ is even. 
Let $A[s,t]\to A[t]$ be the evaluation 
such that $s\mapsto 1$.  Then $\bar{I}= (\bar{f}(i+j-1)-1,\bar{f}(i+j),0)$.  
By Lemma~\ref{lmm:tpoly}, we have that $\bar{f}(i+j-1)$ has constant term 1 and 
so both $\bar{f}(i+j-1)-1$ and $\bar{f}(i+j)$ have no constant term. 
Therefore $I\subseteq (t)$ and thus $\bar{I}\neq A[t]$. Hence $I\neq A[s,t]$  
and so $R\neq \{0\}$.

{\sc Second case:} Suppose $i+j$ is odd.  
Let $A[s,t]\to A[t]$ be the evaluation such 
that $s\mapsto -1$.  Then $\bar{I}=(\bar{f}(i+j-1)-(-1)^{j-1},\bar{f}(i+j),0)$.
Regardless of the parity of $j-1$, both $\bar{f}(i+j-1)-(-1)^{j-1}$ and 
$\bar{f}(i+j)$ are monic by Lemma~\ref{lmm:tpoly}, and 
they share a common factor by Lemma~\ref{lmm:technical}, and thus 
$\bar{I}\neq A[t]$. Hence $I\neq A[s,t]$ and so $R\neq \{0\}$.
\end{proof}

\subsection{Examples}
\label{subsec:exa}

We conclude this section with two examples, the first of which is 
a consequence of Theorem~\ref{thm:1st-main} and is stated in the following
corollary.
\begin{corollary}
Let $A$ be a commutative ring, then $R(A;2,1,1,1)\cong M_2(A)$.
\end{corollary}
\begin{proof}
We know $R(A;2,1,1,1)\cong M_2(A[s,t]/I)$ where 
\[
I=( f(3),f(2)-s^0,s^1-(-1)^1)=( t^2+s,t-1,s+1)=( t-1,s+1)
\]
and so $A[s,t]/I\cong A$ and therefore $R(A;2,1,1,1)\cong M_2(A)$.
\end{proof}
We now consider the specific ${\rats}$-algebra $R({\rats};4,3,1,1)$.  
Again, using Theorem~\ref{thm:1st-main}, we know 
$R({\rats};4,3,1,1)\cong M_2(A[s,t]/I)$ where 
\begin{eqnarray*}
I & = & ( f(7),f(6)-s^2,s^1-(-1)^1)\\
  & = & (t^6+5st^4+6s^2t^2+s^3,t^5+4st^3+3s^2t-s^2,s+1)\\
  & = & (t^3-t^2-2t+1, s+1),
\end{eqnarray*}
since $s$ and $-1$ are in the same coset and the 
$\gcd(t^6-5t^4+6t^2-1,t^5-4t^3+3t-1)=t^3-t^2-2t+1$. Since $t^3-t^2-2t+1$ 
is irreducible over ${\rats}$, then ${\rats}[s,t]/I$ is a field 
extension of ${\rats}$ given by ${\rats}(\lambda)$ where 
$\lambda\in\comps$ satisfies the polynomial equation
$\lambda^3-\lambda^2-2\lambda+1 = 0$.  

While $R({\rats};4,3,1,1)\cong M_2({\rats}(\lambda))$, we 
still have $(4,3,1,1)\notin \mathcal{A}_{\rats}$ since 
there is no non-trivial homomorphism that maps a field 
extension to its base field. This is the case since $1$ must map to $1$ and 
hence ${\rats}$ must map identically to ${\rats}$.

Further, we note that by Theorem~\ref{thm:cyclic} we have for 
$x\in R({\rats};4,3,1,1)$ that $x^6-x^5+x^4-x^3+x^2-x+1=0$
where $x^6-x^5+x^4-x^3+x^2-x+1$ is irreducible over ${\rats}$.
Since each matrix over $\rats$ satisfies its characteristic 
polynomial, this means that the matrix ring $M_6({\rats})$ is the 
smallest possible matrix ring that $R({\rats};4,3,1,1)$ can be 
mapped to non-trivially. On the other hand, since, every 
field extension is a vector space over its base ring and every element 
of a field extension acts linearly on that vector space by multiplication,
then every field extension can be realized as set of matrices of dimension 
equal to the degree of the extension.  Therefore, since $\lambda$ 
satisfies the polynomial equation $\lambda^3-\lambda^2-2\lambda+1 = 0$ of
degree three, we have that ${\rats}(\lambda)$ 
is isomorphic to a subring of $M_3({\rats})$. This gives an explicit
isomorphism of $M_2({\rats}(\lambda))$ into a subring of 
$M_2(M_3({\rats})) = M_6({\rats})$. This means that the smallest
$N$ in~\cite{Geirb} for which there is a non-trivial $\rats$-algebra
homomorphism $R({\rats};4,3,1,1) \rightarrow M_N(\rats)$ is here
$N = 6$. We will see that this observation agrees with
Theorem~\ref{thm:ijArats} in the following Section~\ref{sec:Ak-fields}.

\section{Surjections onto matrix rings over base fields}
\label{sec:Ak-fields}

In the previous section we showed in Theorem~\ref{thm:1st-main}
that if $\gcd(i,j)=1$, then $R(A;i,j,1,1)\cong M_2(A[s,t]/I)$.  
However, we also argued that just because $R(A;i,j,1,1)\cong M_2(L)$ 
for some commutative ring $L$, does not necessarily mean that 
$(i,j,1,1)\in \mathcal{A}_A$ from Definition~\ref{def:ABC}, 
as the second example in the previous 
Subsection~\ref{subsec:exa} showed us.  In this section,
we will restrict our attention to $R(k;i,j,1,1)$ where $k$ is a field 
and investigate the set $\mathcal{A}_k\in{\nats}^4$ for various fields $k$.

Note that $(i,j,1,1)\in \mathcal{A}_{k}$ is equivalent to: ``One can find 
nonzero $2\times 2$ matrices $x, y \in M_2(k)$ satisfying 
$x^iy+yx^j=1$ and $y^2=0$.'' Clearly, if 
$x =\begin{psmallmatrix} 0 & 1 \\ 1 & 0 \end{psmallmatrix}$ and 
$y =\begin{psmallmatrix} 0 & 1 \\ 0 & 0 \end{psmallmatrix}$,
then we have $x^{2u + 1}y+yx^{2v + 1}=1$ and $y^2=0$ for any integers $u,v$,
and so we trivially have the following observations.
\begin{observation}
\label{obs:ij-odd-Ak}
For any field $k$ and odd $i,j\in\nats$, we have 
$(i,j,1,1)\in \mathcal{A}_{k}$.
\end{observation}
For a given field $k$, we would ideally like to determine exactly
for which $i,j\in\nats$ we have $(i,j,1,1)\in \mathcal{A}_{k}$. As this
question is too general to generate any interesting results, we will
focus on the base fields $\rats$ and $\ints_p$ for prime numbers $p\geq 2$.
By left-right symmetry (or stronger, by Theorem~\ref{thm:ijji}) we can
assume $j\geq i$. 

\subsection{Reducing to a four dimensional matrix algebra}
\label{subsec:Sijab}

\begin{definition}
\label{def:de2}
Let $k$ be a field. For $a,b \in k$ define 
\begin{eqnarray*}
S(k;i,j,a,b) & = & R(k;i,j,1,1)/(x^{2}-a x+b) \\              
& = & k\langle x,y : x^{i}y+yx^{j} = 1, \ \ y^{2} = x^{2}- ax + b = 0\rangle.
\end{eqnarray*}
\end{definition}
We have the following lemma:
\begin{lemma}
\label{lmm:Snonzero}
Either $S(k;i,j,a,b)$ is trivial or $S(k;i,j,a,b)\cong _{k}M_{2}(k)$
as $k$-algebras.
\end{lemma}
\begin{proof}
Assume $S(k;i,j,a,b)$ is nonzero. Then as an
image of $R(k;i,j,1,1)$, it must be a nonzero $2\times 2$ matrix algebra. By 
applying the rule $x^{2}=ax-b$ one gets the formulas:
\[
x^{i} = f_{i}x+g_{i}, \ \ x^{j} = f_{j}x+g_{j},
\]
where $f_{l},g_{l}\in k$. Clearly since $S(k;i,j,a,b)$ is nonzero, 
$f_{i},f_{j}\neq 0$ must hold. By putting these expressions
into the equation $x^{i}y+yx^{j}=1$ and then isolating $yx$, one gets an 
expression of the form
\[
yx=-\frac{f_i}{f_j}xy - \frac{g_i+g_j}{f_j}x + \frac{1}{f_j}
\]
which gives a commuting rule for $x$ and $y$. Hence $S(k;i,j,a,b)$ is 
generated by $x$ and $y$ satisfying:
\[
x^{2} = ax -b, \ \ y^{2} = 0, \ \  
yx = -\frac{f_i}{f_j}xy - \frac{g_i+g_j}{f_j}x + \frac{1}{f_j}
\]
which makes $S(k;i,j,a,b)$ a nonzero $k$-algebra spanned 
by $\{ 1,x,y,xy\}$ and can therefore be at most 4-dimensional. 
As a $2\times 2$ matrix algebra over $k$, it must be of dimension
exactly $4$, and so it must be isomorphic to $M_{2}(k)$.
\end{proof}
We will now examine the conditions that $i,j$ and $k$ must satisfy in order 
for $(i,j,1,1)\in {\mathcal{A}}_{k}$. We have already seen 
in Observation~\ref{obs:ij-odd-Ak} that if both $i$ and $j$ are odd
then $(i,j,1,1)\in {\mathcal{A}}_{k}$ for all field $k$,
so we will therefore concentrate on other values of $i$ and $j$. 
We will no longer assume $\gcd(i,j) = 1$.

We note that if $(i,j,1,1)\in {\mathcal{A}}_{k}$ then, since 
$R(k;i,j,1,1)$ can
be mapped onto $M_{2}(k)$ in which every element satisfies its second 
degree characteristic polynomial, there must be $a,b\in k$ such that 
$S(k;i,j,a,b)$ is nonzero, in which case it is isomorphic to $M_{2}(k)$
as a $k$-algebra. 
It therefore is sufficient
to find the conditions $i,j$ and $k$ must satisfy such that there are 
$a ,b\in k$ which make $S(k;i,j,a,b)$ nonzero. Now, if $a ,b\in k$
and $F$ is an extension field of $k$, then 
$S(F;i,j,a,b) = S(k;i,j,a,b)\otimes _{k}F$, 
so $S(k;i,j,a,b)$ is nonzero if and only if $S(F;i,j,a,b)$ is nonzero.
\begin{lemma}
\label{lmm:insep}
For $(i,j)\neq (1,1)$ we have the following:
if $x^{2}-ax+b\in k[x]$ is inseparable with double root $r\in \bar{k}$,
the algebraic closure of $k$,
then: $S(k;i,j,a,b)$ is nonzero if and only if $\charac(k)\mid i+j$,
$\charac(k)\not\,\mid i$ and $r^{j-i}=-1$.
\end{lemma}
\begin{proof}
$S(k;i,j,a,b)$ is nonzero if and only if 
$S(\bar{k};i,j,a,b)$ is nonzero, so we assume either one. 
Since $x^{2}-ax+b = (x-r)^{2} \in \bar{k}[x]$, 
by putting $z=x-r$ we get a new representation of 
$S(\bar{k};i,j,a,b)$ as
\[
\bar{k}\langle y,z\ : 
\ (ir^{i-1}z+r^{i})y+y(jr^{j-1}z+r^{j})=1,\ y^{2}=z^{2}=0\rangle.
\] 
By multiplying the first equation by $y$ left and right, one gets
$ir^{i-1}yzy=y=jr^{j-1}yzy$. Since our algebra is nonzero we get:
\begin{equation}
\label{eqn:irjr}
ir^{i-1}=jr^{j-1} \neq 0.
\end{equation}
By multiplying the same equation by $z$ left and right we get:
\[
(r^{i}+r^{j})zy = z-jr^{j-1}zyz, \ \ 
(r^{i}+r^{j})yz = z-ir^{i-1}zyz.
\]
By (\ref{eqn:irjr}) we get $(r^{i}+r^{j})(yz-zy) = 0$. Since 
$S(\bar{k};i,j,a,b)$ is nonzero and hence $\cong _{k}M_{2}(k)$, it is 
noncommutative so $yz-zy\neq 0$, and so $r^{i}+r^{j}=0$ must hold.
Since $(i,j)\neq (1,1)$ either $i-1$ or $j-1$ is greater than $0$ so by 
(\ref{eqn:irjr}), $r$ cannot be zero. We get therefore $r^{j-i}=-1$ and
$jr^{j-i}=i\neq 0$, so we get the necessary conditions: $r^{j-i}=-1$ and
$i+j$ is divisible by $\charac(k)$ but neither $i$ nor $j$ are.
These conditions are sufficient since if they hold, then one can map:
\[
y\mapsto
\begin{pmatrix}
0 & \frac{1}{ir^{i-1}} \\
0 & 0
\end{pmatrix}
\ 
z\mapsto
\begin{pmatrix}
 0 & 0 \\
 1 & 0
\end{pmatrix}.
\]
Clearly these matrices satisfy the defining equations for 
$S(\bar{k};i,j,a,b)$ in its new representation.
\end{proof}
Next we examine conditions that will make $S(k;i,j,a,b)$ 
nonzero when $x^{2}-ax+b\in k[x]$ is separable with two distinct roots
$r,s\in \bar{k}$. $S(k;i,j,a,b)$ is nonzero if and only if
$S(\bar{k};i,j,a,b)\cong _{k}M_{2}(\bar{k})$. Now if the image of $x$
under this isomorphism satisfies $x^{2}-ax +b=(x-r)(x-s)$, then 
$x$ is mapped to a $2\times 2$
matrix that has $(x-r)(x-s)$ as a minimal polynomial and is therefore
diagonalizible with eigenvalues $r$ and $s$. We may therefore by an inner
isomorphism of $M_{2}(\bar{k})$ assume the $\bar{k}$-algebra isomorphism 
$S(\bar{k};i,j,a,b)\cong M_{2}(\bar{k})$ to have the form
$x\mapsto \begin{psmallmatrix} r & 0 \\ 0 & s \end{psmallmatrix}$ and 
$y\mapsto\begin{psmallmatrix} b_{1\/1} & b_{1\/2} \\ b_{2\/1} & b_{2\/2} 
\end{psmallmatrix}$.
In order for these matrices to satisfy all the defining relations of 
$S(k;i,j,a,b)$ we now only need to examine $x^{i}y+yx^{j}=1$ and $y^{2}=0$.
The first equation gives the conditions:
\[
(r^{i}+r^{j})b_{1\/ 1} =1,  \ \ (r^{i}+s^{j})b_{1\/ 2} =0, \ \ 
(s^{i}+s^{j})b_{2\/ 2} =1, \ \ (r^{j}+s^{i})b_{2\/ 1} =0.
\]
Clearly $r,s\neq 0$.
Now if $b_{1\/ 2}=0$ then, since $y^{2}=0$, we must have 
$b_{1\/ 1}=b_{2\/ 2}=0$ which is impossible.
We have therefore $b_{1\/ 2}\neq 0$. The same holds for $b_{2\/ 1}$, 
so we have necessary
conditions that $r$ and $s$ must satisfy:
\begin{equation}
\label{eqn:r-and-s}
r^{i}+s^{j} = 0\neq r^{i}+r^{j}, \ \ r^{j}+s^{i}  = 0 \neq  s^{i}+s^{j}. 
\end{equation}
If we do have $r,s\in k$ satisfying (\ref{eqn:r-and-s}), then letting
$b_{1\/1}=\frac{1}{r^{i}+r^{j}}, b_{2\/ 2}=\frac{1}{s^{i}+s^{j}},
b_{1\/2}=b_{2\/1}=\frac{\mbox{\bf i}}{r^{i}+r^{j}}$, where 
$\mbox{\bf i}^{2}=-1$, it is easy to check that:
$x\mapsto \begin{psmallmatrix} r & 0 \\ 0 & s \end{psmallmatrix}$,
$y\mapsto \begin{psmallmatrix} b_{1\/1} & b_{1\/2} \\ b_{2\/1} & b_{2\/2}
\end{psmallmatrix}$
is indeed a $\bar{k}$-algebra isomorphism 
$S(\bar{k};i,j,a,b)\cong M_{2}(\bar{k})$. It is therefore sufficient to find
$r,s\in\bar{k}$ satisfying (\ref{eqn:r-and-s}).

Therefore, we are looking for $r,s\in \bar{k}$ satisfying a second degree
polynomial over $k$, so in addition to the necessary conditions of 
(\ref{eqn:r-and-s}) that $r,s\in \bar{k}$ must satisfy, 
we must also have
$r+s,rs\in k$. Clearly these conditions combined are sufficient to make
$S(k;i,j,a,b)$ well defined and nonzero.
\begin{lemma}
\label{lmm:rs}
For $r,s\in \bar{k}^{*}$ (\ref{eqn:r-and-s}) are equivalent to:
\begin{equation}
\label{eqn:rs-condit}
(rs)^{j-i} = 1, \ \ r^{i+j}+(rs)^{i} = 0, r^{j-i} \neq  -1.
\end{equation}
So, $S(k;i,j,a,b) \neq \{0\}$ if and only if there
are $r,s\in\bar{k}^{*}$ satisfying (\ref{eqn:rs-condit}) such
that $r+s,rs\in k$. $\Box$ 
\end{lemma}
By looking at the conditions of Lemma~\ref{lmm:rs}, along with the
condition $r+s,rs \in k$, we see that whether there is a root $r$ of 
$x^{i+j}+\zeta ^{i} \in k[x]$ in $\bar{k}$, where $\zeta$ is a root of 
$x^{j-i}-1$ in $k$, such that $r+\zeta /r \in k$ and $r^{j-i}\neq -1$, 
depends not only on the characteristic of $k$ but also on what kind 
of an extension field of the base fields ($\rats$ or $\ints _{p}$) $k$ is. 
Since by Definition~\ref{def:ABC} we have for all fields that 
$k_{1}\subseteq k_{2}$ implies that 
${\mathcal{A}}_{k_{1}}\subseteq {\mathcal{A}}_{k_{2}}$ and
${\mathcal{B}}_{k_{1}}\subseteq {\mathcal{B}}_{k_{2}}$, it seems natural
to study the initial element in the category of fields with a 
certain characteristic. Hence we will consider the cases 
$k=\rats, \ints _{p}$ for primes $p\geq 2$. Since $-1=1$ in 
$k=\ints _{2}$ we will dispatch that special case first.
\begin{theorem}
\label{thm:AF2}
$(i,j,1,1)\in {\mathcal{A}}_{\ints _{2}}$ if and only if 
\[
(i,j)\equiv \left\{ 
\begin{array}{ll}
(1,1) \pmod{2} \\
(1,2),\ (2,1) \pmod{3}.
\end{array} \right.
\]
\end{theorem}
\begin{proof}
We need to find necessary and sufficient conditions
on $(i,j)$ such that one can find $a ,b\in \ints _{2}$ with 
$S(\ints_2;i,j,a,b)$ nonzero. There are two cases:

{\sc First case:} $x^{2}-a x+b$ is inseparable:
Here by Lemma~\ref{lmm:insep} we must have
both $i$ and $j$ odd numbers, which is also sufficient by
Observation~\ref{obs:ij-odd-Ak}.

{\sc Second case:} $x^{2}-a x+b$ is separable: Here by
Lemma~\ref{lmm:rs} the roots must 
satisfy $rs=1,r^{i+j}=1,r^{j-i}\neq 1$ and $r+s\in \ints _{2}$ and so
$r+1/r =0$ or $1$.

If $r+1/r=0$, then $r^{2}=1$ so $1=r^{i+j}=r^{j-i+2i}=r^{j-i}$,
which is impossible.

If $r+1/r=1$, then $r^{2}+r+1=0$ and so $r^{3}=1$. Since 
$\gcd(x^{3}-1,x^{i+j}-1)=x^{\gcd(3,i+j)}-1$, we must have $i+j$ divisible
by $3$ and $j-i$ not divisible by $3$. Since $i,j\equiv 0,1,2 \pmod{3}$,
we must have $(i,j)\equiv (1,2)\mbox{ or }(2,1) \pmod{3}$. This
is also sufficient since 
$x\mapsto\begin{psmallmatrix} 0 & 1 \\ 1 & 1
\end{psmallmatrix}$, 
$y\mapsto\begin{psmallmatrix} 0 & 1 \\ 0 & 0 \end{psmallmatrix}$
works for this case.
\end{proof}

\subsection{The case for the field of rational numbers}
\label{subsec:rats}

Consider a class of polynomials $f_{n}(x)\in \ints[x]$ defined by:
\begin{equation}
\label{eqn:f_n(x)}
f_{0}(x) = 2, \ \ f_{1}(x) = x \mbox{ and } f_{n+1}(x) = xf_{n}(x)-f_{n-1}(x),
\mbox { for $n\geq 1$. }
\end{equation}
By induction we easily get that
\begin{equation}
\label{eqn:xn-x-n}
x^{n}+x^{-n}=f_{n}(x+x^{-1})
\end{equation}
for every $n\in\nats$.
If $k$ is a base field $\rats$ or $\ints_p$ then through the natural
map $\ints\rightarrow k$ we can view the polynomials $f_n(x)$ in $k[x]$.

Suppose $k\neq \ints_2$ and there are $r,s\in\bar{k}$ and $n\in\nats$ such that
that $r+s\in k$, $rs\in k^{*}$ and $r^{2n}+s^{2n}=0$.
Then by (\ref{eqn:xn-x-n}) we have 
$f_{n}(r/s + s/r) = \frac{r^{2n}+s^{2n}}{(rs)^{n}} =0$ and so
$f_n(x)$ has a root $r/s + s/r\in k$. On the other hand
if $f_n(x)$ has a root $t\in k$, then by (\ref{eqn:xn-x-n})
$t\neq -2$ and we can find $r,s\in \bar{k}$ with $r+s=rs=t+2\in k$
and
\[
0 = f_n(t) = f_n\left(\frac{(r+s)^2}{rs} - 2\right)
= f_n\left(\frac{r}{s} + \frac{s}{r}\right) = \frac{r^{2n}+s^{2s}}{(rs)^n}.
\]
Hence we have the following.
\begin{lemma}
\label{lmm:f_n(x)}
For a base field $k$ there are $r,s\in\bar{k}$ such
that $r+s\in k$, $rs\in k^{*}$ and $r^{2n}+s^{2n}=0$ if and only
if $f_n(x)$ from (\ref{eqn:f_n(x)}) has a root in $k$.
\end{lemma}
We can now prove the main result of this subsection.
\begin{theorem}
\label{thm:ijArats}
$(i,j,1,1)\in {\mathcal{A}}_{\rats }$ if and only if
\[
\begin{array}{lll}
(i,j) & \in & \{ (n,n)\ |\ n\not\equiv 0\pmod{4}\} \mbox{ or} \\
(i,j) & \equiv & \left\{
                       \begin{array}{ll}
                       (1,1)  \pmod{2} \\
                       (1,2),\ (2,1),\ (4,5),\ (5,4) \pmod{6}.
                       \end{array} \right.
\end{array}
\]
\end{theorem}
\begin{proof}
To find out if there are $a ,b\in \rats $ which will
make $S(\rats;i,j,a,b)$, nonzero we may by Lemma~\ref{lmm:insep} assume
$x^{2}-a x+b$ is separable with distinct roots
$r,s \in \overline{\rats }\subseteq \comps$. We now have two cases.

{\sc First case:} $i=j$: Here the necessary and sufficient conditions for 
$S(\rats;i,j,a,b)$ to be nonzero are, by Lemma~\ref{lmm:rs}, 
the existences of $r,s \in \comps$ such that
$r^{i}+s^{i}=0, r+s\in \rats \mbox{ and } rs\in \rats ^{*}$.
Clearly if  $i$ is odd one can let $r=1$ and $s=-1$, so assume 
$i=2n$ to be even.

By Lemma~\ref{lmm:f_n(x)} there are such $r,s\in\comps$ if and only
if $f_n(x)$ has a rational root, which is the case iff
$f_{n}(x-2)$ has a rational root. By the recursive definition of 
$f_{n}$ in (\ref{eqn:f_n(x)}), we see that $f_{n}(x-2)$
is an $n$-th order polynomial with leading 
coefficient $1$ and constant coefficient $(-1)^{n}2$. So a rational number
$a/b$ with $\gcd(a,b)=1$ is a root of $f_{n}(x-2)$ if and only if 
$b\mid 1$ and $a\mid 2$, hence the only possible rational roots of $f_{n}(x-2)$
are $\pm 1$ or $\pm 2$. The only positive
result one finds is $f_{n}(0)$ for which
\[
\mid f_{n}(0)\mid =\left\{
               \begin{array}{lll}
                0 & \mbox{ if $n$ is odd} \\
                2 & \mbox{ if $n$ is even}
               \end{array}
         \right.
\]
holds. We therefore have from this and Observation~\ref{obs:ij-odd-Ak}
that $(i,i,1,1)\in {\mathcal{A}}_{\rats }$ if and 
only if $i\not\equiv 0 \pmod{4}$.

{\sc Second case:}
$i\neq j$: Again by left-right symmetry of $R(\rats;i,j,1,1)$ we may
assume $j>i$. Here the necessary and sufficient conditions for 
$r,s\in\comps $ to fulfill are by Lemma~\ref{lmm:rs} 
\begin{eqnarray}
\label{eqn:rsj-irats} 
(rs)^{j-i}=1, r^{i+j}+(rs)^{i}=0, r^{j-i}\neq -1,\ r+s,rs\in\rats.
\end{eqnarray}
Since $rs\in\rats $ and $(rs)^{j-i}=1$ we have $rs=\pm 1$. Since
both $r$ and $s$ are here roots of unity in $\comps$, then in order for 
$r+s\in \rats $ to hold, $s$ must be either $-r$ or $1/r$.

If $s=-r$ then we get from (\ref{eqn:rsj-irats}) that 
$0=r^{2i}(r^{j-i}+(-1)^{i})$ and hence, $r^{j-i}=(-1)^{i-1}$. Since 
$r^{j-i}\neq -1$ then $i$ must be odd. Also, $1=(-1)^{j-i}r^{2(j-i)}=(-1)^{j-i}$
and so $j$ must be odd as well.
We conclude, what we already knew from Observation~\ref{obs:ij-odd-Ak},
that $(i,j,1,1)\in {\mathcal{A}}_{\rats }$ if 
both $i$ and $j$ are odd.

If $s=1/r$ then (\ref{eqn:rsj-irats}) becomes
\begin{eqnarray}
\label{eqn:r+1/r}
r^{i+j}=-1, r+1/r\in\rats ,r^{j-i}\neq -1.
\end{eqnarray}
By looking at the two first equations of 
(\ref{eqn:r+1/r}) one sees by (\ref{eqn:xn-x-n}) 
that $r+1/r\in [2,2]$ is a rational root of 
$f_{i+j}(x)+2$. On the other hand if $c\in [2,2]$ is a rational 
root of $f_{i+j}(x)+2$
then by putting $r+1/r=c$ one gets $r^{i+j}+r^{-(i+j)}=f_{i+j}(r+1/r)=-2$
and so $r^{i+j}=-1$. So the existence of an $r$ in $\comps $ satisfying the two 
first conditions of (\ref{eqn:r+1/r}) 
is equivalent to the existence of a rational
root $c\in [2,2]$ of $f_{i+j}(x)+2$.

The leading coefficient of $f_{n}(x)+2$ is $1$ and the constant
term is $2$ if $n$ is odd, so the only 
possible rational roots of $f_{n}(x)+2$ in $[2,2]$ when $n$ is odd are
$\pm 1$ or $\pm 2$.

By (\ref{eqn:xn-x-n}) one gets $f_{2n}(x)+2=f_{n}(x)^{2}$ so 
$f_{2n}(x)+2$ has a rational root if and only if $f_{n}(x)$ has one. As we saw above $n$
must be odd and $f_{n}(0)=0$ is the only possibility.

We have therefore that the only possible rational roots of 
$f_{n}(x)+2$ in $[2,2]$ are $0,\pm 1,\pm 2$ and the only positive
results we find are:
\begin{eqnarray*}
f_{n}(0)+2=0   & \Leftrightarrow & n\equiv 2\pmod{4} \\
f_{n}(1)+2=0   & \Leftrightarrow & n\equiv 3\pmod{6} \\
f_{n}(-2)+2=0  & \Leftrightarrow & n\equiv 1\pmod{2}. 
\end{eqnarray*}
We have therefore the following cases for $r+1/r$ being a root of 
$f_{i+j}(x)+2$ to consider:

$r+1/r=0$ and $i+j=4k+2$: Here we have $r^{2}=-1$, so in order
for $r^{j-i}\neq -1$ to hold we must have neither $i$ nor $j$ even and so 
we get no new information.

$r+1/r=1$ and $i+j=6k+3$: Here we have $r^{2}-r+1=0$ and hence
$r^{3}=1$, so in order for $r^{j-i}\neq -1$ to hold we must have neither
$i$ nor $j$ divisible by $3$, hence $(i,j)\equiv (1,2),(2,1),(4,5)\mbox{ 
or }(5,4)\pmod{6}$. It is on the other hand clear that for these 
values $(i,j)$ there is $r\in \comps$ satisfying (\ref{eqn:r+1/r}).

$r+1/r=-2$ and $i+j=$odd: Here we have $r=-1$ and $j-i=$odd so
$r^{j-i}=-1$. This case is impossible, and we have completed the proof.
\end{proof}

\subsection{The case for a general prime number $p\geq 3$}
\label{subsec:pgeq3}

For a prime number $p$ recall the {\em $p$-adic order} or
the {\em $p$-adic valuation} of the integers
$\nu_p : \ints \rightarrow \nats\cup\{\infty\}$ defined by
$\nu_p(n) = \max(\{\nu\in\nats : p^{\nu} | n \})$ for
$n\neq 0$ and $\nu_p(0) = \infty$. 

We now tackle the case $k=\ints _{p}$ for primes $p\geq 3$.
We start with a few useful lemmas.
\begin{lemma}
\label{lmm:Group}
If $G\subseteq k^{*}$ is a multiplicative
subgroup of a field $k$ of characteristic $\neq 2$ 
and $|G|=n$, then $G$ has an element $g$ with $g^{m}=-1$ if and only if
$\nu_2(m) + 1 \leq \nu_2(n)$.
\end{lemma}
\begin{proof} 
As a subgroup of $k^*$ then $G$ and so every subgroup of $G$ is cyclic.
Let $\nu = \nu_2(m)$. If there is a $g\in G$ with $g^{m}=-1$, then 
$h^{2^{\nu}}=-1$ for some power $h$ of $g$. Since the characteristic 
is not 2 one has $h^{2^{\nu+1}}=1$ and $h^{2^{\nu}}\neq 1$. Hence $h$ has 
order $2^{\nu+1}$ in $G$ and so $2^{\nu+1}$ must divide $|G|$, the order
of $G$.

If $2^{\nu+1}$ divides $\mid G\mid =n$, say $n=2^{\nu+1}c$, then let 
$g=\xi ^{c}$ where $G=\langle\xi \rangle$. As an element of a field we have: 
$0=g^{2^{\nu+1}}-1=(g^{2^{\nu}}-1)(g^{2^{\nu}}+1)$
and hence $g^{2^{\nu}}=-1$ must hold and so $g\in G$ satisfies $g^{m}=-1$.
\end{proof}
\begin{lemma}
\label{lmm:Frobenius}
For ${c}\in \ints_p$ and $z\in\overline{\ints_p}^*$ we have
$z + \frac{{c}}{z}\in \ints_p$ if and only if
$z^{p-1} = 1$ or $z^{p+1} = {c}$.
\end{lemma}
\begin{proof}
For $t = z + \frac{{c}}{z}\in \ints_p$ we have $t\in\ints_p$ if and only
if $t^p = t$, which again is equivalent to $(z^{p-1} - 1)(z^{p+1} - {c}) = 0$
\end{proof}
We can now dispatch the case when $i=j$.
\begin{theorem}
\label{thm:iiAZp}
For a prime $p\geq 3$ we have $(i,i,1,1)\in {\mathcal{A}}_{\ints _{p}}$ 
if and only if $\nu_2(p^2-1)\geq \nu_2(i)+2$.
\end{theorem}
\begin{proof}
For an odd $i$ we have by Observation~\ref{obs:ij-odd-Ak}
that $(i,i,1,1)\in {\mathcal{A}}_{\ints _{p}}$ and since $p\geq 3$ 
we have $\nu_2(p^2-1)\geq 2 = \nu_2(i) + 2$. Hence we can for the 
remainder of the proof assume $i$ to be even.

We want to find $a ,b\in \ints _{p}$ such that 
$S(\ints_p;i,i,a,b)$ is nonzero. We have two cases.

{\sc First case:}
$x^{2}-a x+b$ is inseparable: by Lemma~\ref{lmm:insep} we see that 
if $i\neq 1$ then we must have that $p$ divides $2i$ and not $i$, 
which is impossible since $p\geq 3$. Hence $i=1$ must hold which is 
covered in the theorem.

{\sc Second case:}
$x^{2}-a x+b$ is separable: assume the two roots are $r$ and $s$.
By Lemma~\ref{lmm:rs} the necessary and sufficient conditions for 
$S(\ints_p;i,i,a,b)$ to be nonzero are the existences of 
$r,s\in \overline{\ints }_{p}$ such that:
\begin{equation}
\label{eqn:rsZp*}
r^{i}+s^{i}=0, r+s\in \ints _{p}, rs\in \ints _{p}^{*}.
\end{equation}
We will show that these conditions are equivalent to the existence of 
$\gamma \in \overline{\ints }_{p}$
such that:
\begin{equation}
\label{eqn:gammaZp}
\gamma ^{i}=-1 \mbox{ and } \gamma +\gamma ^{-1} \in \ints _{p}.
\end{equation}
That (\ref{eqn:rsZp*}) implies (\ref{eqn:gammaZp}) can clearly
be gotten by putting $\gamma =r/s$. Then $\gamma ^{i} +1 = 0$ and 
$\gamma +\gamma ^{-1} =r/s+s/r =\frac{(r+s)^{2}}{rs}-2 \in \ints _{p}$.

The other implication we get by putting $r=\gamma +1$ and
$s=\gamma ^{-1} +1$. Then we have: $r^{i}+s^{i}= 
(\gamma +1)^{i}+(\gamma ^{-1}+1)^{i} = (\gamma ^{-1}+1)^{i}(\gamma ^{i}+1)
=0$ and $r+s=rs=\gamma +\gamma ^{-1} +2\in \ints _{p}$ where $i$ is even,
$\gamma\neq -1$ and so $\gamma + \gamma^{-1} + 2 \neq 0$, and therefore
$r+s, rs\in \ints_p^{*}$.

It suffices therefore to find the conditions for $i$ and $p$ such
that there exists a $\gamma \in \overline{\ints }_{p}$ satisfying 
(\ref{eqn:gammaZp}).

By Lemma~\ref{lmm:Frobenius} $\gamma +\gamma ^{-1} \in \ints _{p}$ is 
equivalent to $\gamma$ satisfying $(\gamma ^{p-1}-1)(\gamma ^{p+1}-1) = 0$,
and so either 
$\gamma \in \ints_{p}^{*}$ or 
$\gamma \in G \subseteq \overline{\ints }_{p}^{*}$, 
where $G$ is the cyclic group of order $p+1$ formed by all the roots of 
$x^{p+1}-1 \in \ints _{p}[x]$.

In order for $\ints_{p}$ or $G$ to contain an element $\gamma $
such that $\gamma ^{i}=-1$, it is by Lemma~\ref{lmm:Group}
necessary and sufficient that either $\nu_2(p-1)\geq \nu_2(i)+1$ or
$\nu_2(p+1)\geq \nu_2(i)+1$. Since $\gcd(p-1,p+1)=2$, this is 
equivalent to $\nu_2(p^2-1)\geq \nu_2(i)+2$.
\end{proof}
We will now conclude the article by finding necessary and sufficient 
conditions for $(i,j,1,1)\in {\mathcal{A}}_{\ints_{p}}$ to hold for 
primes $p\geq 3$,
together with a couple of corollaries. 
The proof is elementary and based on ``case-chasing'' using group theory
and congruences.
\begin{theorem}
\label{thm:main-AZp}
For a prime $p\geq 3$ we have $(i,j,1,1)\in {\mathcal{A}}_{\ints _{p}}$ 
only in the following cases: 
(I) $\nu_2(j-i) > \nu_2(p-1)\geq \nu_2(i+j)$,
(II) $\nu_2(j-i) = \nu_2(p-1)\neq \nu_2(i+j)$,
(III) $\nu_2(j-i) < \nu_2(p-1)$ and $\gcd(j-i,p-1)$ is not an odd multiple
of $\gcd(j+i,p-1)$,
(IV) $\nu_2(j-i) < \nu_2(p-1)$, $p|i+j$ and $p\not|i$,
(V) $\nu_2(j - ip) < \nu_2(p+1) + \min(\nu_2(j-i),\nu_2(p-1))$ and
$j-i$ is not an odd multiple of $\gcd (j-ip,(p+1)(j-i),p^2-1)$.
\end{theorem}
{\sc Remark:} Note that in the case of $i=j$ we have $\nu_2(j-i) =\infty$
and so cases (II), (III) and (IV) do not occur. Case (I) reduces then
to $\nu_2(p-1)\geq \nu_2(i) + 1$ and case (V) to 
$\nu_2(p+1)\geq \nu_2(i) + 1$. We therefore see that 
Theorem~\ref{thm:iiAZp} is covered by the above Theorem~\ref{thm:main-AZp}.
\begin{proof}[Theorem~\ref{thm:main-AZp}]
By Theorem~\ref{thm:iiAZp} and by left-right symmetry of $R(k;i,j,1,1)$ 
we can assume $j>i$. 

{\sc Convention:} For convenience in writing this proof we
define $d=\gcd(j-i,p-1)$ and $e=\gcd(j+i,p-1)$.
Also, for the remainder of this proof we will be using
$x,y$ and $z$ for integer variables in congruence equations and not
as generators of our noncommutative algebra $R(A;i,j,m,n)$.

To find out whether there are $a ,b\in\ints _{p}$ 
with $S(\ints_p;i,j,a,b)$ nonzero, we have, as in the proof
of Theorem~\ref{thm:iiAZp}, two cases:

{\sc First case:} 
$x^{2}-a x+b$ is inseparable with a double root $r\in \ints _{p}$.
If $r=0$ we must have $i=j=1$ which is not the case here, so $r\neq 0$. In
this case by Lemma~\ref{lmm:insep} we must have $p$ divide $i+j$ but not 
$i$ and $r^{j-i}=-1$ which by Lemma~\ref{lmm:Group} can only occur 
when $\nu_2(j-i) <\nu_2(p-1)$ so we have gotten case $(IV)$ in the theorem.

{\sc Second case:}
$x^{2}-a x+b$ is separable with roots $r,s \in \overline{\ints }_{p}$.
By Lemma~\ref{lmm:rs} $r$ and $s$ must satisfy the following conditions:
\begin{equation}
\label{eqn:rsj-i}
(rs)^{j-i}=1, \ \ r^{i+j}+(rs)^{i}=0, \ \ r^{j-i}\neq-1, \ \ r+s,rs\in\ints_{p}. 
\end{equation}
Since $rs\in \ints _{p}$ here we have $(rs)^{p-1}=1$ and so if 
$d=\gcd(j-i,p-1)$, then $rs$ must be some power of a primitive $d$-th root 
of unity in $\ints _{p}$, that is $rs=\rho ^{{\beta}y}$ where 
${\beta}=\frac{p-1}{d}$, $y$ is 
some integer in $\ints $ and $\langle\rho \rangle=\ints _{p}^{*}$. So 
the conditions (\ref{eqn:rsj-i}) become
$r^{i+j}=-\rho ^{{\beta}yi}$, $r^{j-i}\neq -1$,
$r+\frac{\rho ^{y{\beta}}}{r}\in \ints _{p}$.
By Lemma~\ref{lmm:Frobenius} $r+\frac{\rho ^{y{\beta}}}{r}\in \ints _{p}$
is equivalent to $(r^{p-1}-1)(r^{p+1}-\rho ^{y{\beta}}) = 0$, so either 
$r\in \ints _{p}$, or $r\in \overline{\ints }_{p}$ satisfying 
$r^{p+1}=\rho ^{y{\beta}}$ where $y\in\ints$ and so we have two sub-cases.

{\sc First sub-case:}
$r\in \ints _{p}$. Since $r\neq 0$ we 
have that $r=\rho ^{x}$ for some $x\in \ints$. So here we get the conditions
\[
\rho ^{x(i+j)}=-\rho ^{{\beta}yi}=\rho^{\frac{p-1}{2}+{\beta}yi} \ \mbox{ and }\ 
\rho ^{x(j-i)}\neq \rho ^{\frac{p-1}{2}}.
\]
Since $\rho $ is the generator of $\ints _{p}^{*}$ it has period 
$p-1$ so we get the conditions for $x,y\in\ints$:
\begin{eqnarray*}
x(i+j) & \equiv & \frac{p-1}{2}+{\beta}yi \pmod{p-1} \\
x(j-i) & \not\equiv & \frac{p-1}{2} \pmod{p-1} 
\end{eqnarray*}
which is the same as whether one can find $x,y,z \in \ints $ such that:
\begin{eqnarray*}
2x(i+j) - 2{\beta}yi & = & (2z+1)(p-1) \\
2x(j-i)       &\neq & (\mbox{odd\# })(p-1).
\end{eqnarray*}
Since $p-1={\beta}d$, let $j-i={\alpha}d$. Then $\gcd({\alpha},{\beta})=1$ and 
$2{\beta}i={\beta}(i+j)-{\alpha}(p-1)$ and so we get:
\begin{equation}
\label{eqn:ijxny}
(i+j)(2x-{\beta}y)=(2z-{\alpha}y+1)(p-1), \ \ 2x(j-i)\neq(\mbox{odd\# })(p-1).
\end{equation}
If $\nu_2(j-i) =\nu_2(p-1)$ then by definition, both ${\alpha}$ and ${\beta}$ 
are odd, say ${\alpha}=2{\alpha}'-1$ and ${\beta}=2{\beta}'-1$. Let us now 
change the variables $x$ and $z$ to $x'=x-{\beta}'y$ and $z'=z-{\alpha}'y$. 
Now the first equation of (\ref{eqn:ijxny}) becomes:
\[
(i+j)(2x'+y)=(2z'+y+1)(p-1)
\]
which clearly has a solution $x',y,z'\in \ints$ and hence also 
$x,y,z \in \ints $ in this case, if and
only if $\nu_2(i+j) \neq \nu_2(p-1) $, because $2x'+y$
and $2z'+y+1$ have distinct parity. When
$\nu_2(j-i) =\nu_2(p-1) \neq \nu_2(i+j)$ then 
$2x(j-i)\neq (\mbox{odd\#})(p-1)$ so we have gotten the case $(II)$ 
in the theorem.

If $\nu_2(j-i) >\nu_2(p-1) $ then ${\alpha}$ is even and ${\beta}$ is odd, 
say ${\alpha}=2{\alpha}'$ and ${\beta}=2{\beta}'-1$. Again let 
$x'=x-{\beta}'y$ and $z'=z-{\alpha}'y$. Here the first equation of 
(\ref{eqn:ijxny}) becomes:
\[
(i+j)(2x'+y)=(2z'+1)(p-1)
\]
which clearly has solutions $x',y,z'\in \ints$ and hence also 
$x,y,z \in \ints $ if and only if $\nu_2(p-1) \geq \nu_2(i+j) $. In this case
$2x(j-i)$ is never an odd multiple of $(p-1)$, so we have gotten here case 
$(I)$ in the theorem.

If $\nu_2(j-i) <\nu_2(p-1) $ then ${\alpha}$ is odd and ${\beta}$ is even, 
say ${\alpha}=2{\alpha}'-1$ and ${\beta}=2{\beta}'$. Let $x',z'$ be as 
before and we get from (\ref{eqn:ijxny}) that 
\begin{equation}
\label{eqn:s-odd-mult}
(i+j)(2x')=(2z'+y+1)(p-1).
\end{equation}
Here we clearly can always find solutions $x',y,z'\in \ints $
to (\ref{eqn:s-odd-mult}) and hence also corresponding
solutions in $x,y,z\in \ints$. Assume now that for every such solution 
$x,y,z\in \ints$ we have $2x(j-i)=(\mbox{odd\# })(p-1)$. We first notice that 
this condition is equivalent to $x=(\mbox{odd\# }){\beta}'$, so we are here 
assuming that every solution $x,y,z\in \ints $ to 
$(i+j)(2x-{\beta}y)=(2z-{\alpha}y+1)(p-1)$ has
$x=(\mbox{odd\# }){\beta}'$. Let us write $i+j=et$ and $p-1=se$ where 
$e=\gcd(i+j,p-1)$.
Then since $\gcd(s,t)=1$ every solution $x,y,z$ must satisfy:
\begin{eqnarray*}
2x-{\beta}y & = & {\ell}s \\
2z-{\alpha}y+1 & = & {\ell}t 
\end{eqnarray*}
for some ${\ell}\in \ints $. Since $x'=x-{\beta}'y$ and 
$z'=z-{\alpha}'y$, then we assume that every solution $x',y,z'$ to 
\begin{eqnarray*}
2x' & = & {\ell}s \\
2z'+y+1 & = & {\ell}t 
\end{eqnarray*}
has $x'=(\mbox{odd\# } - y){\beta}'$, i.e. whenever $2x'={\ell}s$ then 
$x'=(\mbox{odd\# } - ({\ell}t-2z'-1)){\beta}'=(2c_{{\ell}}-{\ell}t){\beta}'$ 
for some $c_{{\ell}}\in \ints$. So in particular when $x'=s$ (i.e. ${\ell}=2$) 
we have $x'=(2c_{2}-2t){\beta}'=(c_{2}-t){\beta}$ and hence $s$ is even, 
which means $t$ is odd since $\gcd(s,t)=1$. We therefore have a solution
$2x'=s$ (i.e. ${\ell}=1$) and from this we get
$s=(2c_{1}-t){\beta}=(\mbox{odd\# }){\beta}$, an odd multiple of $\beta$.

Suppose now that $s$ is an odd multiple of ${\beta}$,
say $s = (2w+1){\beta}$ for some $w\in\ints$, then 
$x'={\ell}(2w+1){\beta}'$ and 
$y = {\ell}t-(2z'+1)$. Since $t$ is odd, ${\ell}(2w+1+t) - 2z'$ is
even and so ${\ell}(2w+1) + y$ is odd, say $2w'+1$, and so
$x' = {\ell}(2w+1){\beta}' = (2w'+1 - y){\beta}'$. In summary, if
$s$ is an odd multiple of ${\beta}$, then every solution
$x',y,z'$ to (\ref{eqn:s-odd-mult}) has $x'=(\mbox{odd\# }-y){\beta}'$.

We have therefore in the case $\nu_2(j-i)  < \nu_2(p-1) $ that 
$s$ not being an odd multiple of ${\beta}$ is the necessary and sufficient 
condition, and since this is equivalent to $d\neq (\mbox{odd\# })e$ we 
get case $(III)$ in the theorem.

{\sc Second sub-case:}
$r^{p+1}=\rho ^{y{\beta}}$, where $y\in \ints $.
By (\ref{eqn:rsj-i}) we want to find the conditions for the existence of 
$r\in \overline{\ints }_{p}$ such that:
\begin{equation}
\label{eqn:rrho}
r^{i+j}=-\rho ^{{\beta}yi} , r^{p+1}=\rho ^{{\beta}y} ,r^{j-i} \neq -1.
\end{equation}
Since $r\neq 0$ and the cyclic group $\ints_p^*$ of order $p-1$ is
generated by $\rho$, so $\ints_p^* = \langle\rho\rangle$, then 
(\ref{eqn:rrho}) is equivalent to 
\[
r^{j-ip}=-1 , r^{(p+1)d}=1 , r^{j-i}\neq-1.
\]
The condition $r^{(p+1)d}=1$ is equivalent to $r\in 
G\subseteq \overline{\ints }_{p}^{*}$ where $G$ is the finite cyclic 
group of order
$(p+1)d$ formed by all the roots of $x^{(p+1)d}-1$. By lemma~\ref{lmm:Group} 
we must have $\nu_2((p+1)d)>\nu_2(j-ip)$. 

Let us now assume that each solution  of $r^{(p+1)d}=1 , r^{j-ip}=-1$
satisfies $r^{j-i}=-1$. Let $\xi $ be the generator of $G$. Since $\xi$ is an 
element of $\overline{\ints }_{p}$, we have 
$0=\xi ^{(p+1)d}-1=\left( \xi^{\left( \frac{p+1}{2}\right) d} -1\right) \left( \xi
^{\left( \frac{p+1}{2}\right) d} +1\right)$ 
so $\xi ^{\left(\frac{p+1}{2}\right)d} =-1$. Since now $r$ is a power of 
$\xi $ we are in fact assuming that every $x$ satisfying 
$x(j-ip)\equiv \left( \frac{p+1}{2}\right)d \pmod{(p+1)d}$ also 
satisfies
$x(j-i) \equiv \left( \frac{p+1}{2}\right)d \pmod{(p+1)d}$. This is 
equivalent to assuming that every solution $x,y\in \ints$ to 
\begin{equation}
\label{eqn:2x2y+1}
2x(j-ip)=(2y+1)(p+1)d
\end{equation}
has $2x(j-i)$ as an odd multiple of $(p+1)d$.

Since $\nu_2((p+1)d)>\nu_2(j-ip)$ we have 
$q=\gcd\left( j-ip,\frac{p+1}{2}d\right) =\gcd(j-ip,(p+1)d)$. 
Write $j-ip=qu$ and $\frac{p+1}{2}d=qv$. Then every
solution to (\ref{eqn:2x2y+1}) must be:
\begin{eqnarray*}
x & = & {\ell}v \\
2y + 1 & = & {\ell}u 
\end{eqnarray*}
where ${\ell}\in \ints$. Clearly ${\ell}$ can only be odd. We have in 
particular for the solution $x=v$ and $y=\frac{u-1}{2}$ (i.e. ${\ell}=1$) 
that $2x(j-i)=(\mbox{odd\# })2qv$; that is to say $j-i=(\mbox{odd\# })q$.

On the other hand, one can easily see that if 
$j-i=(\mbox{odd\# })q$ then every solution 
to (\ref{eqn:2x2y+1}) satisfies $2x(j-i)=(\mbox{odd\# })(p+1)d$. We have 
therefore finally that one can find $r\in \overline{\ints }_{P}$ satisfying 
(\ref{eqn:rrho}) if and only if $\nu_2((p+1)d)>\nu_2(j-ip)$ and 
$j-i\neq (\mbox{odd\# })\gcd(j-ip,(p+1)d)$. Since
$\nu_2((p+1)d)>\nu_2(j-ip) = \nu_2(p+1) = \min(\nu_2(j-i),\nu_2(p-1))$
we finally have case $(V)$ in the theorem.
\end{proof}
{\sc Remarks:} (i) The case that $i$ and $j$ are odd numbers is 
included in the theorem since $i,j$ both odd is the same as saying
``Exactly one of the numbers $\nu_2(i+j) ,\nu_2(j-i) \in \nats $ 
is equal to 1 and the other is 2 or greater.''
(ii) The case $(V)$ in the theorem doesn't look symmetrical in $i,j$,
but it is since $j-ip\equiv i-jp \pmod{(p+1)d}$ means that 
$\nu_2((p+1)d)>\nu_2(j-ip)$ if and only if
$\nu_2((p+1)d)>\nu_2(i-jp)$. 

We conclude the article by two corollaries of Theorem~\ref{thm:main-AZp}
for particular primes $p$. Since $\nu_2(p-1) = 1$ is equivalent
to $p\equiv -1 \pmod{4}$, we have by Theorem~\ref{thm:main-AZp} and
the Chinese Remainder Theorem the following corollary.
\begin{corollary}
\label{cor:p=-1mod4}
For a prime $p$ of the form $p=2^a(2b+1)-1$ where $a\geq 2$, in particular
for each Mersenne prime of the form $p=2^q-1$ for a prime $q$, 
we have $(i,j,1,1)\in {\mathcal{A}}_{\ints_{p}}$ 
only in the following cases:
(I) $i$ and $j$ are both odd,
(II) $i$ and $j$ have distinct parity and 
$\gcd(j-i,p-1)$ is not an odd multiple of $\gcd(j+i,p-1)$,
(III) 
\[
(i,j)\equiv (\ell(p+1), (p+1)(p - \ell) + p),\ 
((p+1)(p - \ell) + p, \ell(p+1)) \pmod{2p}, \mbox{ for } 
\ell\in \{1,\ldots, p-1\},
\]
(IV) $\nu_2(j - ip) \leq a$ and 
$j-i$ is not an odd multiple of $\gcd (j-ip,(p+1)(j-i),p^2-1)$.
\end{corollary}
Further, for $p=3$ we note that the second condition (II) in the 
above Corollary~\ref{cor:p=-1mod4} cannot occur, and 
since $p^2-1=8$, a pure power of $2$, the last condition (IV) 
becomes ``$\nu_2(j - 3i) \leq 2$ and 
$j-i$ is not an odd multiple of $\gcd (j-3i,4(j-i),8)$.'' We now briefly
translate this condition by considering two cases: 
(a) if $\nu_2(i)\neq\nu_2(j)$, then 
$i = 2^{\nu}i'$ and $j = 2^{\nu}j'$ where $i'$ and $j'$ have distinct parity
and $\nu_2(j-i) = \nu_2(j-3i) = \nu\leq 2$. Since both $j'-i'$ and $j'-3i'$
are odd and $\gcd(j'-3i',4(j'-i'),2^{3-\nu}) = 1$, then $j-i = 2^{\nu}(j'-i')$
is always an odd multiple of 
$2^{\nu}\gcd(j'-3i',4(j'-i'),2^{3-\nu}) = \gcd(j-3i,4(j-i),8)$.
(b) If $\nu_2(i) = \nu_2(j) = \nu$, then $i = 2^{\nu}(2i'+1)$
and $j = 2^{\nu}(2j'+1)$, and hence $j-i = 2^{\nu+1}(j'-i')$
and $j-3i = 2^{\nu+1}(j'-3i'-1)$. By the first condition (I) 
we can assume $\nu \geq 1$, and since $\nu + \nu_2(j'-3i'-1)\leq 1$
we can assume $\nu = 1$ and $j'-3i'-1$ to be odd.
We note that $\gcd(j-3i,4(j-i),8) = 2^{\nu+1}\gcd(j'-3i'-1,4(j'-i'),2^{2-\nu})
=2^{\nu+1}\gcd(j'-3i'-1,2^{2-\nu})$. Since $j'-i'$ and $j'-3i'-1$ have
distinct parity, $j'-i'$ is never an odd multiple of 
$\gcd(j'-3i'-1,2^{2-\nu}) = 1$. We therefore 
have from Corollary~\ref{cor:p=-1mod4} the following. 
\begin{corollary}
\label{cor:p=3}
$(i,j,1,1)\in {\mathcal{A}}_{\ints _{3}}$ if and only if 
\[
(i,j)\equiv \left\{ 
\begin{array}{ll}
(1,1) \pmod{2}, \\
(1,2),\ (2,1),\ (4,5),\ (5,4) \pmod{6}, \\
(2,2),\ (6,6) \pmod{8}.
\end{array} \right.
\]
\end{corollary}

\subsection*{Acknowledgments}  

Sincere thanks to the anonymous referee for spotting some
well-hidden typos and for improving the exposition of this article.

\bibliographystyle{amsalpha}
\bibliography{Matrix-bib}

\flushright{\today}

\end{document}